\documentclass[review]{elsarticle}

\usepackage{amsfonts}
\usepackage{amssymb}
\usepackage{mathrsfs}
\usepackage{latexsym}
\usepackage{pgf,tikz}
\usetikzlibrary{arrows}
\usetikzlibrary[patterns]
\usepackage{lineno}
\modulolinenumbers[5]

\usepackage{setspace}

\usepackage{graphicx} 
\graphicspath{{Figures/}} 
\usepackage{enumitem} 
\usepackage{lipsum} 
\usepackage{subfig} 
\usepackage{amsmath,amssymb,amsthm} 
\usepackage{varioref} 
\usepackage{bm}
\usepackage{bbold}
\usepackage{scalerel}
\usepackage[makeroom]{cancel}
\usepackage{mathtools}
\usepackage{float}

\usepackage[text={6.5in,9in}]{geometry}
\theoremstyle{plain} 
\newtheorem{theorem}{Theorem}

\theoremstyle{proposition}
\newtheorem{proposition}{Proposition}[section]

\newcommand {\be}{\begin{equation}} 
\newcommand {\ee}{\end{equation}}
\newcommand{\tr}{\mathrm{tr}}
\newcommand{\divg}{\mathrm{div}\;}
\newcommand{\Vnorm}[1]{||#1||_{\mathcal{V}}}
\newcommand{\intO}[1]{\int_{\Omega}#1 \;dx}
\newcommand{\intG}[2]{\int_{\Gamma_#1}#2 \;ds}
\newcommand{\sumOe}{\sum_{\Omega_e \in \mathcal{T}_h}}
\newcommand{\sumGiD}{\sum_{E\in\Gamma_{iD}}}
\newcommand{\sumintOe}[1]{\sum_{\Omega_e \in \mathcal{T}_h}\int_{\Omega_e}#1 \;dx}
\newcommand{\sumintGiD}[1]{\sum_{E\in\Gamma_{iD}}\frac{1}{h_E}\int_{E}#1 \;ds}
\newcommand{\sumintG}[2]{\sum_{E\in\Gamma_#1}\dfrac{1}{h_E}\int_E#2 \;ds}
\newcommand{\intOe}{\int_{\Omega_e}}

\newcommand{\DGnorm}[1]{||#1||_{DG}}

\begin{document}
\begin{frontmatter}
\title{Discontinuous Galerkin approximations for near-incompressible and near-inextensible transversely isotropic bodies}
\author[mainaddress]{B.J. Grieshaber}
\ead{beverley.grieshaber@uct.ac.za}

\author[mainaddress]{F. Rasolofoson\corref{mycorrespondingauthor}}
\cortext[mycorrespondingauthor]{Corresponding author}
\ead{rslfar002@myuct.ac.za}

\author[mainaddress]{B.D. Reddy}
\ead{daya.reddy@uct.ac.za}

\address[mainaddress]{Department of Mathematics and Applied Mathematics \\
and Centre for Research in Computational and Applied Mechanics\\
University of Cape Town, South Africa}

\begin{abstract}
This work studies discontinuous Galerkin (DG) approximations of the boundary value problem for homogeneous transversely isotropic linear elastic bodies. 
Low-order approximations on triangles are adopted, with the use of three interior penalty DG methods, viz. nonsymmetric, symmetric and incomplete. 
It is known that these methods are uniformly convergent in the incompressible limit for the isotropic case. 
This work focuses on behaviour in the inextensible limit for transverse isotropy. 
An error estimate suggests the possibility of extensional locking, a feature that is confirmed by numerical experiments. 
Under-integration of the extensional edge terms is proposed as a remedy. 
This modification is shown to lead to an error estimate that is consistent with locking-free behaviour. Numerical tests confirm the uniformly convergent behaviour, at an optimal rate, of the under-integrated scheme. 
\end{abstract}


\end{frontmatter}
\section{Introduction}
Anisotropic materials have a wide range of applications, e.g. in the geological domain, or in biomechanical systems such as the myocardium, brain stem, ligaments, and tendons \cite{Exadaktylos2001,Royer2011,shahi2013}.
Different types of anisotropy along with their mechanical restrictions are presented in \cite{Lubarda-Chen2008,Ting1996,Zubov2016}.
In this work, we are particularly interested in transversely isotropic materials, which play a central role in theories describing the behaviour of fibre-reinforced composite materials \cite{Hayes-Horgan1974,Hayes-Horgan1975,Pipkin1979}.

When elastic materials are internally constrained, the associated displacement-based finite element approximations exhibit poor performance in the form of poor coarse-mesh approximations, as well as locking, in which they do not converge uniformly with respect to the constraint parameters. 
The incompressibility constraint has been widely studied in this context (see, for example, \cite{Brezzi-Fortin1991}), while  behaviour in the inextensible limit has received less attention. 

It has been shown that locking can be avoided using high-order elements \cite{Babuska1992}, though low-order approximations remain of high interest. 
Low-order discontinuous Galerkin approximations with linear approximations on triangles are uniformly convergent in the incompressible limit \cite{Wihler}. 
The corresponding problem using bi- or trilinear approximations on quadrilaterals and hexahedra displays locking, which may be overcome by selective under-integration of edge terms involving the relevant Lam\'e parameter \cite{Beverley}.
 
A range of mixed methods have been shown to be uniformly convergent for near-incompressibility (see for example 
\cite{Boffi-Brezzi-Fortin2008} and the references therein), while the works \cite{DLRW2006,LRW2006} provide a unified treatment of convergent approaches using two- or three-field approximations. 
Near-inextensibility is studied computationally in the work \cite{Auricchio-Scalet-Wriggers2017}, using Lagrange multiplier and perturbed Lagrangian approaches.
There have also been a number of computational investigations of transversely isotropic and inextensible behaviour for large-displacement problems
\cite{Wriggers-Schroeder-Auricchio2016,Zdunek-Rachowicz-Eriksson2016,Zdunek-Rachowicz2017a,Zdunek-Rachowicz2017b}.

Theoretical studies of anisotropic elastic behaviour include the work \cite{Arnold-Falk1987}, in which conditions for well-posedness are established for a Hellinger-Reissner formulation, and \cite{Hayes-Horgan1974,Hayes-Horgan1975}, in which conditions for uniqueness and stability are established.
 
In recent work \cite{RGR}, the authors have studied the well-posedness of boundary value problems involving transversely isotropic elastic materials. 
They have also investigated theoretically and computationally the use of conforming finite element approximations, paying attention to both near-incompressibility and near-inextensibility. 
It is found in that work that, for low-order quadrilaterals, selective under-integration of volumetric and extensional terms serves to render the schemes locking-free. 
Further related work has recently been reported in \cite{Reddy-vanHuyssteen2018}, on a virtual element formulation for transverse isotropy: this formulation is shown to be robust and locking-free in the inextensional limit, for both constant and variable fibre directions.

The subject of this work is a study of DG approximations for transversely isotropic linear elasticity. 
The focus of the work is on three interior penalty methods: symmetric, nonsymmetric, and incomplete. 
We draw on earlier work on DG formulations for elliptic problems in \cite{Arnold2002} and elasticity in \cite{Beverley,Hansbo2002,Hansbo2003,Wihler}, in addressing the problem of developing discrete formulations that are uniformly convergent in the incompressible and inextensible limits. 

With the use of low-order triangles, the problem for isotropic elasticity is uniformly convergent in the incompressible limit \cite{Hansbo2002,Wihler}. 
For bi- or trilinear approximations on quadrilaterals and hexahedra, however, it is known \cite{Beverley} that uniform convergence requires the use of under-integration of the edge terms in the formulation. 
Both of the corresponding error estimates rely ultimately on an a priori estimate presented in \cite{Brenner1992}. 
For the problem studied here, it is shown numerically that the DG methods in their original formulation result in locking behaviour in the inextensional limit, a problem that is resolved by underintegrating the relevant edge terms. 
There does not exist an a priori estimate for transverse isotropy analogous to that in \cite{Brenner1992} for the isotropic problem, but the error estimate for the formulation with under-integration has a structure similar to that of the isotropic a priori estimate, and therefore is consistent with the uniformly convergent behaviour that is observed numerically. 

The outline of the rest of this work is as follows. 
In Sections \ref{Sec:TI} and \ref{Sec:Governing_eq}, we present the governing equations and weak formulation for problems of transversely isotropic linear elasticity, with a summary of conditions for well-posedness. 
The DG formulations are introduced in Section \ref{Sec:DGFEM},  their well-posedness established, and an a priori error bound derived. 
The likelihood of extensional locking is deduced from the error estimate, and an alternative formulation, based on selective under-integration, is introduced and analyzed in Section \ref{Sec:UI}. 
The similarity of the resulting error bound to that for isotropic elasticity suggests the locking-free behaviour of this formulation. 
Such behaviour is confirmed in Section \ref{Sec:Numerical}, in which numerical results are presented for two model problems.   
The work concludes with a summary of results and a discussion of possible future work.

\section{Transversely isotropic materials}\label{Sec:TI}
For a transversely isotropic linearly elastic material with fibre direction given by the unit vector $\bm{a}$, the elasticity tensor is given by \cite{Lubarda-Chen2008,Spencer1982}
\be\label{tangent_moduli}
\mathbb{C} = \lambda \bm{I}\otimes\bm{I} + 2\mu_t \mathbb{I} + \beta \bm{M}\otimes\bm{M} + \alpha (\bm{I}\otimes\bm{M}+\bm{M}\otimes\bm{I}) + \gamma \mathbb{M}\,.
\ee
Here $\bm{I}$ is the second-order identity tensor, $\mathbb{I}$ is the fourth-order identity tensor, $\bm{M} = \bm{a}\otimes \bm{a}$, and $\mathbb{M}$ is the fourth-order tensor defined by
\be
\mathbb{M}\bm{R} =   \bm{M}\bm{R} + \bm{R}\bm{M}\qquad\mbox{for any second-order tensor}\ \bm{R}\,. 
\label{mathcalM}
\ee
$\lambda$ denotes the first Lam\'e parameter, the shear modulus in the plane of isotropy is $\mu_t$, $\mu_l$ is the shear modulus along the fibre direction, and 
\be
\gamma = 2(\mu_l - \mu_t).
\label{gamma}
\ee
The further material constants $\alpha$ and $\beta$ do not have a direct interpretation, though it will be seen that $\beta \rightarrow \infty$ in the inextensible limit.

The corresponding linear stress-strain relation for small deformations is then 
\begin{align}
	\bm\sigma &= \mathbb{C}\bm{\varepsilon}\nonumber\\
			  &= \lambda(\tr\bm\varepsilon)\bm{I} + 2\mu_t \bm\varepsilon + \beta (\bm{M}:\bm\varepsilon)\bm{M} + \alpha ((\bm{M}:\bm\varepsilon)\bm{I} + (\tr\bm\varepsilon)\bm{M}) + \gamma (\bm\varepsilon\bm{M} + \bm{M}\bm\varepsilon),\label{cauchy_tensor}
\end{align}
in which $\bm\sigma$ and $\bm\varepsilon$ denote the stress and the infinitesimal strain tensors; $\mbox{tr}\,\bm{\varepsilon}$ denotes the trace of $\bm{\varepsilon}$, and $\bm{M}:\bm{\varepsilon} = \bm{\varepsilon}\bm{a}\cdot\bm{a}$, obtained from the definition of $\bm{M}$, gives the strain in the direction of $\bm{a}.$
The special case of an isotropic material is recovered by setting $\alpha = \beta = 0$ and $\mu_l = \mu_t$.

We can write the expressions of these five material parameters in terms of five physically meaningful constants, viz.
$E_t$: Young's modulus in the transverse direction; $E_l$: Young's modulus in the fibre direction; and
$\nu_t$ and $\nu_l$: Poisson's ratios for the transverse strain with respect to the fibre direction and the plane normal to it, respectively.
The remaining constants are the two shear moduli $\mu_t$ and $\mu_l$, and one may further define  $\mu_t$ by 
\be
\mu_t = \dfrac{E_t}{2(1+\nu_t)}.
\label{GT}
\ee
Henceforth, we set
\be\label{cpq}
E_l = p E_t \hspace{0.3cm}\text{ and } \hspace{0.3cm} \mu_l = q \mu_t\,.
\ee
Thus $p$ measures the stiffness in the fibre direction relative to that in the plane of isotropy.
We then have
\begin{subequations}\label{material_parameters}
\begin{align}
\mu_l &= \dfrac{qE_t}{2(1+\nu_t)},\\
\lambda &= \dfrac{(\nu_tp+\nu_l^2)}{(1+\nu_t)((1-\nu_t)p-2\nu_l^2)}E_t,\\[10pt]
\alpha &= \dfrac{(\nu_l-\nu_t+\nu_t\nu_l)p-\nu_l^2}{((1+\nu_t)((1-\nu_t)p-2\nu_l^2)}E_t,\\[10pt]
\beta &= \dfrac{(1-\nu_t^2)p^2 + (-2\nu_t\nu_l+2q\nu_t-2\nu_l+1-2q)p - (1-4q)\nu_l^2}{(1+\nu_t)((1-\nu_t)p-2\nu_l^2)}E_t.\label{beta}
\end{align}
\end{subequations}
Later, we will consider the special case in which
\be\label{eq_nu_mu}
\nu_l=\nu_t=\nu \quad \text{and} \quad \mu_l = \mu_t, \quad \text{that is} \quad q=1.
\ee
In this way, we will focus on behaviour in relation to three independent parameters, viz. $(\lambda, \alpha, \beta)$ or $(E_t, E_l, \nu)$, rather than the full set of five parameters.
The expressions given by equations \eqref{material_parameters} then become
\begin{subequations}\label{material_parameters_SC}
\begin{align}
\dfrac{\mu_l}{E_t} &= \dfrac{1}{2(1+\nu)},\\
\dfrac{\lambda}{E_t} &= \dfrac{\nu (p+\nu)}{(1+\nu)\big((1-\nu)p-2\nu^2\big)},\\[10pt]
\dfrac{\alpha}{E_t} &= \dfrac{\nu^2(p-1)}{(1+\nu)\big((1-\nu)p-2\nu^2\big)},\\[10pt]
\dfrac{\beta}{E_t} &= \dfrac{(p-1)\big((1-\nu^2)p - 3\nu^2\big)}{(1+\nu)\big((1-\nu)p-2\nu^2\big)}.
\end{align}
\end{subequations}
Necessary and sufficient conditions on the material constants for the elasticity tensor to be pointwise stable are that the subdeterminants of $\mathbb{C}$ when expressed in matrix form be positive (\cite{Noble1969}, Theorem 12.6) This leads to the conditions
\begin{subequations}\label{PS_conds}
\begin{align}
&\mu_t>0,\, \mu_l>0,\label{l_block}\\
&\lambda+\mu_t>0,\label{det_2}\\
\text{and }\,& (\lambda+\mu_t)(\lambda+\mu_t+2\alpha+\beta+2\gamma) - (\lambda+\alpha)^2>0.\label{det_3}
\end{align}
\end{subequations}
Using the expressions in \eqref{material_parameters}, these conditions can be rewritten in terms of the engineering constants as follows (see also \cite{Exadaktylos2001,Lai2009}):
\begin{subequations}\label{ps_conds_EC}
\begin{align}
& E_t>0,\,\mu_t>0,\,\mu_l>0,\label{psc_1}\\
& p>\nu_l^2,\label{psc_2}\\
& (1-\nu_t)p-2\nu_l^2>0.\label{psc_3}
\end{align}
\end{subequations}
\section{Governing equations and weak formulation}\label{Sec:Governing_eq}
Consider a transversely isotropic elastic body occupying a bounded domain $\Omega \subset \mathbb{R}^d,\,d=\{2,3\}$, with boundary $\Gamma = \Gamma_D \cup \Gamma_N$ having exterior unit normal $\bm{n}$. Here $\Gamma_D$ is the Dirichlet boundary, $\Gamma_N$ the Neumann boundary, and $\Gamma_D\cap \Gamma_N = \emptyset$.
The equilibrium equation is
\be\label{equilibrium}
	-\divg \bm\sigma(\bm{u}) = \bm{f}
\ee
and the boundary conditions are
\begin{subequations}\label{BC}
	\begin{align}
		\bm{u} = \bm{g} & \text{ on } \Gamma_D,\label{Dirichlet}\\
		\bm\sigma(\bm{u})\bm{n} = \bm{t} & \text{ on } \Gamma_N.\label{Neumann}
	\end{align}
\end{subequations}
Here $\bm\sigma$ is the Cauchy stress tensor defined by equation \eqref{cauchy_tensor}, $\bm{u}$ is the displacement vector, $\bm{f}$ is the body force, $\bm{g}$ a prescribed displacement, and $\bm{t}$ a prescribed surface traction.

We denote by $\mathcal{L}^2(\Omega)$ the space of (equivalence classes of) functions which are square-integrable on $\Omega$, and by $\mathcal{H}^1(\Omega)$ the space of (equivalence classes of) functions which, together with their generalized first derivatives, are in $\mathcal{L}^2(\Omega)$. We set
\[
\mathcal{V} = \{\bm{u} \in \big[\mathcal{H}^1(\Omega)\big]^d; \;\bm{u} = \bm{0} \; \text{on} \; \Gamma_D\},
\]
which is endowed with the norm
\[
\Vnorm{\cdot} = \|\cdot\|_{[\mathcal{H}^1(\Omega)]^d}.
\]
To take into account the non-homogeneous boundary condition \eqref{Dirichlet}, we define the function $\bm{u}_g \in [\mathcal{H}^1(\Omega)]^d$ such that $\bm{u}_g = \bm{g}$ on $\Gamma_D$, and the bilinear form $a(\cdot,\cdot)$ and linear functional $l(\cdot)$ by
\begin{subequations}\label{def}
\begin{align}
a: [\mathcal{H}^1(\Omega)]^d \times [\mathcal{H}^1(\Omega)]^d \rightarrow \mathbb{R},& \hspace{1cm} a(\bm{u},\bm{v}) = \intO{\bm{\sigma}(\bm{u}):\bm{\varepsilon}(\bm{v})},\label{adef}\\
l: [\mathcal{H}^1(\Omega)]^d \rightarrow \mathbb{R},& \hspace{1cm} l(\bm{v}) = \intO{\bm{f}\cdot\bm{v}} + \intG{N}{\bm{t}\cdot\bm{v}} - a(\bm{u}_g,\bm{v}).\label{ldef}
\end{align}
\end{subequations}
The weak form of the problem is then as follows: given $\bm{f} \in [\mathcal{L}^2(\Omega)]^d$ and $\bm{t} \in [\mathcal{H}^{1/2}(\Gamma_N)]^d$ \footnote{$\mathcal{H}^{1/2}(\Gamma_N)$ is the space of (equivalence classes of) functions defined on part of or an entire boundary $\Gamma$, with the property that the quantity
$\|v\|_{1/2,\Gamma} := \int_\Gamma\int_\Gamma \left(|v|^2 + \frac{|v(\bm{x}) - v(\bm{y})|^2}{|\bm{x} - \bm{y}|^2}\right)\,ds(\bm{x})ds(\bm{y})$
is bounded, and hence defines a norm on $H^{1/2}(\Gamma)$.},
find $\bm{U} \in [\mathcal{H}^1(\Omega)]^d$ such that $\bm{U}=\bm{u}+\bm{u}_g, \bm{u} \in \mathcal{V}$, and
\be\label{weak_form}
a(\bm{u},\bm{v}) = l(\bm{v}) \hspace{1cm} \forall\ \bm{v}\in \mathcal{V}.
\ee
We write the bilinear form as
\[
a(\bm{u},\bm{v}) = a^{iso}(\bm{u},\bm{v}) + a^{ti}(\bm{u},\bm{v}),
\]
where
\begin{subequations}\label{a_form}
\begin{align}
a^{iso}(\bm{u},\bm{v})
				&= \lambda \intO{(\nabla \cdot \bm{u})(\nabla \cdot \bm{v})}
					+ 2 \mu_t \intO{\bm{\varepsilon}(\bm{u}):\bm{\varepsilon}(\bm{v})},\label{a_iso}\\
a^{ti}(\bm{u},\bm{v})
					&=\alpha \intO{\big((\bm{M}:\bm{\varepsilon}(\bm{u}))(\nabla \cdot \bm{v}) + (\nabla \cdot \bm{u})(\bm{M}:\bm{\varepsilon}(\bm{v}))\big)}
				+ \beta \intO{(\bm{M}:\bm{\varepsilon}(\bm{u}))(\bm{M}:\bm{\varepsilon}(\bm{v}))}\nonumber\\
					&+ \gamma \intO{\big(\bm{\varepsilon}(\bm{u})\bm{M}:\bm{\varepsilon}(\bm{v}) + \bm{M}\bm{\varepsilon}(\bm{u}):\bm{\varepsilon}(\bm{v})\big)}.\label{a_TI}
\end{align}
\end{subequations}
Note that $a(\cdot,\cdot)$ is symmetric and corresponds to a positive definite operator, as shown in \cite{RGR}, meaning that the weak problem is well-posed. 
Regarding regularity of the solution, for $\Omega$ bounded and convex, for example, and $\bm{f}$ and $\bm{t}$ defined as in the line after \eqref{def}, there exists a unique solution $\bm{u} \in [H^2(\Omega)]^d$ (see for example \cite{Necas1967,Brenner1992}).

We introduce the following notation for the relevant norm and seminorms:
\[
\|\cdot\|_{0,*} = \|\cdot\|_{[\mathcal{L}^2(*)]^d}, \quad |\cdot|_{1,*} = |\cdot|_{[\mathcal{H}^1(*)]^d}, \quad \text{and} \quad |\cdot|_{2,*} = |\cdot|_{[\mathcal{H}^2(*)]^d}.
\]
\section{Discontinuous Galerkin finite element approximations}\label{Sec:DGFEM}
Suppose that $\Omega$ is polygonal (in $\mathbb{R}^2$) or polyhedral (in $\mathbb{R}^3$), partitioned into a triangular/tetrahedral conforming mesh comprising $n_e$ disjoint subdomains $\Omega_e$, where each $\Omega_e$ has  boundary $\partial\Omega_e$, consisting of edges/faces $E$, and with outward unit normal $\bm{n}_e$.
Denote by $\mathcal{T}_h := \{\Omega_e\}_e$ the set of all elements, and let $T(\Gamma) := \prod_{\Omega_e \in \mathcal{T}_h} [\mathcal{L}^2(\partial\Omega_e)]$.
Define $h_E:= \mathrm{diam}(E)$, $h_e:= \mathrm{diam}(\Omega_e)$, and $h := \mathrm{max}_{\Omega_e\in{\cal T}_h} \{h_e\}$.
We define the discrete space $\mathcal{V}^h_{\scaleto{DG}{4pt}} \subset \mathcal{V}$ by
\be\label{Vh_DG_space}
\mathcal{V}^h_{\scaleto{DG}{4pt}} := \{ \bm{v} \in [\mathcal{L}^2(\Omega)]^d : \bm{v}|_{\Omega_e} \in [\mathcal{P}_1(\Omega_e)]^d, \;\; \forall \;\Omega_e \in \mathcal{T}_h\},
\ee
where $\mathcal{P}_1(\Omega)$ is the space of polynomials on $\Omega$ of maximum total degree $1$.

Let $\Omega_i, \Omega_e \in \mathcal{T}_h$ be two elements sharing an interior edge $E = \partial\Omega_i \cap \partial\Omega_e$,
and let $\bm{n}$ be the outward normal to $\Omega_i$.
Then for any vector quantity $\bm{v} \in [T(\Gamma)]^d$ and any second order tensor $\boldsymbol{\tau} \in [T(\Gamma)]^{d\times d}$, we define the jumps
\[
\lfloor \bm{v} \rfloor = (\bm{v}_i - \bm{v}_e) \otimes \bm{n},\hspace{.5cm}
\lfloor \boldsymbol{\tau} \rfloor = (\boldsymbol{\tau}_i - \boldsymbol{\tau}_e) \bm{n},\hspace{.5cm}
[\bm{v}] = (\bm{v}_i - \bm{v}_e) \cdot \bm{n},
\]
and the averages
\[
\lbrace\bm{v}\rbrace = \frac{1}{2} (\bm{v}_i + \bm{v}_e),\hspace{.5cm}
\lbrace\boldsymbol{\tau}\rbrace = \frac{1}{2} (\boldsymbol{\tau}_i + \boldsymbol{\tau}_e),
\]
where subscripts $i$ and $e$ denote values on the elements $\Omega_i$ and $\Omega_e$, respectively.
The space $\mathcal{V}^h_{\scaleto{DG}{4pt}}$ is endowed with the norm (see for example \cite{Beverley,Wihler})
\be\label{DG_norm}
||\bm{u}||_{DG}^2 := \sumOe \|\bm{\varepsilon}(\bm{u})\|_{0,\Omega_e}^2 + \dfrac{1}{2}\sumGiD \dfrac{1}{h_E}\|\lfloor\bm{u}\rfloor\|_{0,E}^2,
\ee
where $\Gamma_{iD}$ is the union of all interior edges and all Dirichlet boundary edges.

The general Interior Penalty Discontinuous Galerkin (IPDG) formulation is as follows \cite{Beverley,Wihler}:
for all $\bm{v} \in \mathcal{V}_{\scaleto{DG}{4pt}}^h$, find $\bm{u}_h \in \mathcal{V}_{\scaleto{DG}{4pt}}^h$ such that
\be\label{DGweakform}
a_h(\bm{u}_h,\bm{v}) = l_h(\bm{v}),
\ee
where
\begin{align}\label{ah_DG}
a_h(\bm{u},\bm{v}) =& 
\sumintOe{\bm\sigma(\bm{u}):\bm\varepsilon(\bm{v})}
-  \sum_{E\in\Gamma_{iD}}\int_E\{\bm\sigma(\bm{u})\}:\lfloor\bm{v}\rfloor\,ds
+ \theta  \sum_{E\in\Gamma_{iD}}\int_E \lfloor\bm{u}\rfloor:\{\bm\sigma(\bm{v})\}\,ds\nonumber\\
&+ \sum_{E\in\Gamma_{iD}}\int_E \frac{k}{h_E}\mathbb{C}\lfloor\bm{u}\rfloor:\lfloor\bm{v}\rfloor \,ds,
\end{align}
and
\begin{align}\label{lh_DG}
l_h(\bm{v}) &=
\sumintOe{\bm{f}\cdot\bm{v}}
+ \sum_{E\in\Gamma_{N}}\int_E\bm{t}\cdot\bm{v}\,ds
+ \theta \sum_{E\in\Gamma_{D}}\int_E(\bm{g}\otimes\bm{n}):\bm\sigma(\bm{v})\,ds\nonumber\\
&+  \sum_{E\in\Gamma_{D}}\int_E \frac{k}{h_E}\mathbb{C}(\bm{g}\otimes\bm{n}):(\bm{v}\otimes\bm{n}) \,ds.
\end{align}
Here the elasticity tensor $\mathbb{C}$ is as defined in \eqref{tangent_moduli}, the stress tensor $\bm\sigma$ is as defined in \eqref{cauchy_tensor}, $k$ is a non-negative stabilization parameter, and
$\theta$ is a switch that distinguishes the three methods, ($\theta = 1$ for the Nonsymmetric Interior Penalty Galerkin (NIPG) method, $\theta = -1$ for Symmetric Interior Penalty Galerkin (SIPG), and $\theta = 0$ for Incomplete Interior Penalty Galerkin (IIPG)).

We note that it is possible to assign different stabilization parameters to the different terms in the stabilization.
Thus alternatively, we can define the bilinear form and the linear functional
\begin{align}
\hat{a}_h(\bm{u},\bm{v}) =& 
\sumintOe{\bm\sigma(\bm{u}):\bm\varepsilon(\bm{v})}
- \sumintGiD{\{\bm\sigma(\bm{u})\}:\lfloor\bm{v}\rfloor}
+ \theta \sumintGiD{\lfloor\bm{u}\rfloor:\{\bm\sigma(\bm{v})\}}\nonumber\\
&+ k_\lambda \lambda\sumintGiD{[\bm{u}] [\bm{v}]}
+ k_\mu \mu_t\sumintGiD{ \lfloor \bm{u} \rfloor : \lfloor \bm{v} \rfloor}\nonumber\\
&+ k_\alpha \alpha\sumintGiD{\Big((\bm{M}:\lfloor\bm{u}\rfloor)[\bm{v}] + [\bm{u}](\bm{M}:\lfloor\bm{v}\rfloor)\Big)}\label{ah_DG_k}\\
&+ k_\beta \beta\sumintGiD{ (\bm{M}:\lfloor\bm{u}\rfloor)(\bm{M}:\lfloor\bm{v}\rfloor)}\nonumber\\
&+ k_\gamma \gamma\sumintGiD{ (\lfloor\bm{u}\rfloor\bm{M}+\bm{M}\lfloor\bm{u}\rfloor) : \lfloor \bm{v} \rfloor},\nonumber
\end{align}
and
\begin{align}
\hat{l}_h(\bm{v}) & =
\sumintOe{\bm{f}\cdot\bm{v}}
+ \sumintG{N}{\bm{t}\cdot\bm{v}}
+ \theta \sumintG{D}{(\bm{g}\otimes\bm{n}):\bm\sigma(\bm{v})}\nonumber\\
&+ k_\lambda \lambda\sumintG{D}{ (\bm{g}\cdot\bm{n}) (\mathbf{v}\cdot\bm{n})}
+ k_\mu \mu_t\sumintG{D}{\bm{g} \cdot \bm{v}}\nonumber\\
&+ k_\alpha \alpha\sumintG{D}{\Big((\bm{M}:\bm{g}\otimes\bm{n})(\bm{v}\cdot\bm{n}) + (\bm{g}\cdot\bm{n})(\bm{M}:\bm{v}\otimes\bm{n})\Big)}\label{lh_DG_k}\\
&+ k_\beta \beta\sumintG{D}{(\bm{M}:\bm{g}\otimes\bm{n})(\bm{M}:\bm{v}\otimes\bm{n})} \nonumber\\
& + k_\gamma \gamma\sumintG{D}{\Big((\bm{g}\otimes\bm{n})\bm{M}+\bm{M}(\bm{g}\otimes\bm{n})\Big) : \bm{v}\otimes\bm{n}},\nonumber
\end{align}
where $k_\mu,k_\lambda,k_\alpha,k_\beta,\;\text{and}\;k_\gamma$ are non-negative stabilization parameters.

We confine attention to homogeneous bodies, so that the fibre direction $\bm{a}$ is constant.

Some useful bounds involving $\bm{v} \in [\mathcal{H}^1(\Omega_e)]^d$ and $\phi \in [\mathcal{L}^2(\Omega_e)]^d$ are:
\begin{subequations}
\begin{align}
&\|\bm{v}\|_{0,E} \leq C h_e^{-1/2} \|\bm{v}\|_{0,\Omega_e},\label{bound_e}\\
&\|\bm{v}\|_{0,\partial\Omega_e} \leq C\left( h_e^{-1/2} \|\bm{v}\|_{0,\Omega_e} + h_e^{1/2} |\bm{v}|_{1,\Omega_e}\right),\label{bound_a}\\
& \sum_{E\in\Gamma_{iD}} \dfrac{1}{h_E} \|\lfloor\bm{v}\rfloor\|^2_{0,E} \leq \sum_{\Omega_e\in\mathcal{T}_h} \sum_{E\in\partial\Omega_e^{iD}} \dfrac{2}{h_E} \|\bm{v}\|^2_{0,E},\label{bound_d}\\
& \sum_{E\in\Gamma_{iD}} h_E \|\{\phi\}\|^2_{0,E} \leq C \sum_{\Omega_e\in\mathcal{T}_h} h_e \|\phi\|^2_{0,\partial\Omega_e^{iD}},\label{bound_b}\\
& \sum_{E\in\Gamma_{iD}} \dfrac{1}{h_E} \|\{\phi\}\|^2_{0,E} \leq \sum_{\Omega_e\in\mathcal{T}_h} \sum_{E\in\partial\Omega_e^{iD}} \dfrac{1}{h_E} \|\phi\|^2_{0,E}\label{bound_c}.
\end{align}
\end{subequations}

%
\subsection{Consistency}\label{sec:consistency}
Given the exact solution $\bm{u}$, and assuming that $\bm{u} \in [\mathcal{H}^2(\Omega)]^d$, the problem is consistent if, for any $\bm{v} \in \mathcal{V}_{\scaleto{DG}{4pt}}^h$, 
$$a_h(\bm{u},\bm{v}) - l_h(\bm{v}) = 0.$$
Using the fact that $\lfloor\bm{u}\rfloor_{|_E} = [\bm{u}]_{|_E} = \lfloor\bm\sigma(\bm{u})\rfloor_{|_E} = 0 \;\;\forall\;E\in\Gamma_i$ (set of all interior edges),
and $\lfloor\bm{u}\rfloor = \bm{g}\otimes\bm{n}$ and $[\bm{u}] = \bm{g}\cdot \bm{n}$ on $\Gamma_D$,
the proof of consistency is straightforward and may be carried out in a single argument for all three cases.
%
\subsection{Coercivity}\label{sec:coercivity}
The bilinear form $a_h(\cdot,\cdot)$ is coercive if, for any $\bm{v} \in \mathcal{V}_{\scaleto{DG}{4pt}}^h$,
\[
a_h(\bm{v},\bm{v}) \geq K \|\bm{v}\|_{DG}^2,
\]
where $K> 0$ is a constant. To prove coercivity, each IP method will be investigated separately, as different approaches are used for each of them.

The bilinear form $a_h(\cdot,\cdot)$ defined by \eqref{ah_DG}, which uses a fixed stabilization parameter $k$ for all stabilization terms, is used here.
%
\paragraph{NIPG ($\theta = 1$)}
For the non-symmetric IPDG case, for any $\bm{v} \in \mathcal{V}_{\scaleto{DG}{4pt}}^h$, we have
\begin{align}
a_h(\bm{v},\bm{v}) &= \sumintOe{\bm\sigma(\bm{v}):\bm\varepsilon(\bm{v})}
	+ \sum_{E\in\Gamma_{iD}}\int_E \frac{k}{h_E}\mathbb{C}\lfloor\bm{v}\rfloor:\lfloor\bm{v}\rfloor\,ds.\label{NIPG_coerc}
\end{align}
Using matrix notation, with $\underline{\bm\sigma} = [\sigma_{11}\ \sigma_{22}\ \sigma_{33}\ \sigma_{12}\ \sigma_{13}\ \sigma_{23}]^T$, 
$\underline{\bm\varepsilon} = [\varepsilon_{11}\ \varepsilon_{22}\ \varepsilon_{33}\ 2\varepsilon_{12}\ 2\varepsilon_{13}\ 2\varepsilon_{23}]^T$, and $\underline{\mathbb{C}}$ the corresponding $6\times 6$ matrix, we have
\[
\bm\sigma(\bm{v}):\bm\varepsilon(\bm{v}) = \underline{\bm\varepsilon}^T \underline{\mathbb{C}}\, \underline{\bm\varepsilon}.
\]
Given that $\underline{\mathbb{C}}$ is symmetric and positive definite, it has a set of six positive eigenvalues $\Lambda_i, 1\leq i \leq 6$, and a corresponding set of mutually orthogonal eigenvectors $\bm\xi^i, 1\leq i \leq 6$.
Thus, $\underline{\mathbb{C}}$ can be written as
\be\label{C_diag}
\underline{\mathbb{C}} = \underline{\mathbb{Q}}^T\underline{\mathbb{D}}\,\underline{\mathbb{Q}},
\ee
in which $\underline{\mathbb{D}}$ is a diagonal matrix whose diagonal components are the eigenvalues $\Lambda_i$, and $\underline{\mathbb{Q}}$ is an orthogonal matrix whose columns are the eigenvectors $\bm\xi^i$; that is,
\be\label{D_Q}
\mathbb{D}_{ij} = \Lambda_i\delta_{ij} \quad \text{and} \quad \mathbb{Q}_{ij} = \xi^i_j.
\ee
For any vector $\underline{\bm\varepsilon}$, we define
\be\label{eta}
\underline{\bm\eta} := \underline{\mathbb{Q}}\,\underline{\bm\varepsilon},
\ee
then
\be\label{epsilon}
\underline{\bm\varepsilon} = \underline{\mathbb{Q}}^T\underline{\bm\eta}.
\ee
Therefore,
\begin{align*}
\underline{\bm\varepsilon}^T\underline{\mathbb{C}}\,\underline{\bm\varepsilon}
&= \underline{\bm\eta}^T\underline{\mathbb{Q}} (\underline{\mathbb{Q}}^T\underline{\mathbb{D}}\,\underline{\mathbb{Q}}) \underline{\mathbb{Q}}^T\underline{\bm\eta}\\
&= \underline{\bm\eta}^T \underline{\mathbb{D}} \underline{\bm\eta}\\
&= {\eta}_{i} {\mathbb{D}}_{ij} {\eta}_j\\
&= \sum_{i,j} {\eta}_{i} \Lambda_i\delta_{ij} {\eta}_j\\
&= \sum_{i} {\eta}_{i}^2 \Lambda_i\\
&\geq \Lambda_{min} \sum_{i} {\eta}_{i}^2\\
&= \Lambda_{min} |\underline{\bm\eta}|^2\\
&= \Lambda_{min} |\underline{\bm\varepsilon}|^2,
\end{align*}
in which
\[
\Lambda_{min} = \min \{\Lambda_i, 1\leq i \leq 6\}.
\]
Hence, we have, for $\bm{v} \in \mathcal{V}_{\scaleto{DG}{4pt}}^h$
\begin{align*}
\intOe{\bm\sigma(\bm{v}):\bm\varepsilon(\bm{v}) \;dx }
&\geq \Lambda_{min} \intOe{|\underline{\bm\varepsilon}(\bm{v})|^2} \;dx\\
& = \Lambda_{min} \|\bm\varepsilon(\bm{v})\|^2_{0,\Omega_e}.
\end{align*}
By choosing $\bm{\varepsilon} = \lfloor \bm{v} \rfloor$ in \eqref{epsilon}, we have for $\bm{v} \in \mathcal{V}_{\scaleto{DG}{4pt}}^h$
\[
\sum_{E\in\Gamma_{iD}}\int_E \frac{k}{h_E}\mathbb{C}\lfloor\bm{v}\rfloor:\lfloor\bm{v}\rfloor\,ds \geq k \Lambda_{min} \sum_{E\in\Gamma_{iD}} \dfrac{1}{h_E} \|\lfloor \bm{v} \rfloor\|_{0,E}^2.
\]
Therefore,
\begin{align}
a_h(\bm{v},\bm{v}) &\geq \Lambda_{min} \sumOe\|\bm\varepsilon(\bm{v})\|^2_{0,\Omega_e} + k \Lambda_{min} \sum_{E\in\Gamma_{iD}}\frac{1}{h_E} ||\lfloor\bm{v}\rfloor||^2_{0,E}\nonumber\\
&\geq K \DGnorm{\bm{v}}^2\label{ah_NIPG}
\end{align}
with
\be\label{NIPG_coercivity_bound}
K = \Lambda_{min} \min\left\{1, 2k\right\}>0.
\ee
We conclude that the bilinear form is coercive for the NIPG case.
%
\paragraph{SIPG ($\theta = -1$)}
For the SIPG case, the bilinear form can be written as follows, for any $\bm{v} \in \mathcal{V}_{\scaleto{DG}{4pt}}^h$:
\begin{align}
a_h(\bm{v},\bm{v}) &= \sumOe\intOe \mathbb{C}\bm\varepsilon(\bm{v}):\bm{\varepsilon}(\bm{v}) \; dx
	- 2\sum_{E\in\Gamma_{iD}} \int_E \{\mathbb{C}\bm\varepsilon(\bm{v})\}:\lfloor\bm{v}\rfloor\,ds
	+ \sum_{E\in\Gamma_{iD}}\int_E \frac{k}{h_E}\mathbb{C}\lfloor\bm{v}\rfloor:\lfloor\bm{v}\rfloor\,ds.\label{SIPG_coerc}
\end{align}
Note that since the elasticity tensor $\mathbb{C}$ possesses major symmetries and is positive definite, there exists a unique square root $\mathbb{C}^{1/2}$ such that,
for any second-order tensors $\bm{A}$ and $\bm{B}$, we have:
\be\label{C_equality}
\mathbb{C}\bm{A} : \bm{B} = \mathbb{C}^{1/2}\bm{A} : \mathbb{C}^{1/2} \bm{B}.
\ee
We have
\begin{align*}
\mathscr{A} :=& - 2\sum_{E\in\Gamma_{iD}} \int_E \{\mathbb{C}\bm\varepsilon(\bm{v})\}:\lfloor\bm{v}\rfloor\,ds\\
=& -2\sum_{E\in\Gamma_{iD}} \int_E \mathbb{C}^{1/2}\{\bm\varepsilon(\bm{v})\}:\mathbb{C}^{1/2}\lfloor\bm{v}\rfloor\,ds.
\end{align*}
Thus,
\begin{align*}
\mathscr{A}
\geq& -2 \left(\sum_{E\in\Gamma_{iD}} \|h_E^{1/2}\mathbb{C}^{1/2}\{\bm\varepsilon(\bm{v})\}\|_{0,E}^2\right)^{1/2} \left(\sum_{E\in\Gamma_{iD}} \|h_E^{-1/2}\mathbb{C}^{1/2}\lfloor\bm{v}\rfloor\|_{0,E}^2\right)^{1/2}\\
\geq& -2 C \left(\sumOe h_e\|\mathbb{C}^{1/2}{\bm\varepsilon(\bm{v})}\|_{0,\partial\Omega_e^{iD}}^2\right)^{1/2} \left(\sum_{E\in\Gamma_{iD}} \|h_E^{-1/2}\mathbb{C}^{1/2}\lfloor\bm{v}\rfloor\|_{0,E}^2\right)^{1/2}\quad {\color{gray} (\text{using \eqref{bound_b}})}\\
\geq& -2 C\left(\sumOe \|\mathbb{C}^{1/2}{\bm\varepsilon(\bm{v})}\|_{0,\Omega_e}^2\right)^{1/2} \left(\sum_{E\in\Gamma_{iD}} \|h_E^{-1/2}\mathbb{C}^{1/2}\lfloor\bm{v}\rfloor\|_{0,E}^2\right)^{1/2}\quad {\color{gray} (\text{using \eqref{bound_e}})}\\
\geq& - C\epsilon\left(\sumOe \|\mathbb{C}^{1/2}{\bm\varepsilon(\bm{v})}\|_{0,\Omega_e}^2\right)
-\dfrac{1}{\epsilon}\left(\sum_{E\in\Gamma_{iD}} \|h_E^{-1/2}\mathbb{C}^{1/2}\lfloor\bm{v}\rfloor\|_{0,E}^2\right)\\
=& -\epsilon C \sumOe \intOe \mathbb{C}^{1/2}{\bm\varepsilon(\bm{v})}:\mathbb{C}^{1/2}{\bm\varepsilon(\bm{v})} \,dx
-\dfrac{1}{\epsilon}\sum_{E\in\Gamma_{iD}}\dfrac{1}{h_E}\int_E \mathbb{C}^{1/2}\lfloor\bm{v}\rfloor:\mathbb{C}^{1/2}\lfloor\bm{v}\rfloor\,ds\\
=& -\epsilon C \sumOe \intOe \mathbb{C}{\bm\varepsilon(\bm{v})}:{\bm\varepsilon(\bm{v})} \,dx
-\dfrac{1}{\epsilon}\sum_{E\in\Gamma_{iD}}\dfrac{1}{h_E}\int_E \mathbb{C}\lfloor\bm{v}\rfloor:\lfloor\bm{v}\rfloor\,ds.
\end{align*}
Therefore, we have
\begin{align*}
a_h(\bm{v},\bm{v}) &\geq (1-\epsilon C) \sumOe \intOe \mathbb{C}{\bm\varepsilon(\bm{v})}:{\bm\varepsilon(\bm{v})} \,dx
+ \left(k-\dfrac{1}{\epsilon}\right)\sum_{E\in\Gamma_{iD}}\dfrac{1}{h_E}\int_E \mathbb{C}\lfloor\bm{v}\rfloor:\lfloor\bm{v}\rfloor\,ds.
\end{align*}
We set the coefficient in the first term to a constant $m>0$,
giving $\epsilon = \dfrac{1-m}{C}$, and restrict $m$ to $0 < m < 1$ to ensure $\epsilon>0$.

Then
\begin{align*}
2\left(k-\dfrac{1}{\epsilon}\right) \geq m &\Leftrightarrow k-\dfrac{C}{1-m} \geq \dfrac{m}{2}\\
&\Leftrightarrow k \geq \dfrac{m}{2} + \dfrac{C}{1-m}.
\end{align*}
With this choice of $k$, we have
\begin{align}
a_h(\bm{v},\bm{v})
&\geq m\left(\sumOe \intOe \mathbb{C}{\bm\varepsilon(\bm{v})}:{\bm\varepsilon(\bm{v})} \,dx
+ \dfrac{1}{2}\sum_{E\in\Gamma_{iD}}\dfrac{1}{h_E}\int_E \mathbb{C}\lfloor\bm{v}\rfloor:\lfloor\bm{v}\rfloor\,ds\right)\nonumber\\
&\geq m \Lambda_{min} \left(\sumOe \|\bm\varepsilon(\bm{v})\|_{0,\Omega_e}^2
+ \dfrac{1}{2} \sum_{E\in\Gamma_{iD}} \dfrac{1}{h_E} \|\lfloor\bm{v}\rfloor\|_{0,E}^2\right)\nonumber\\
&= K \|\bm{v}\|_{DG}^2,\label{ah_SIPG}
\end{align}
with
\be\label{SIPG_coercivity_bound}
K = m \Lambda_{min}.
\ee
Therefore, we can conclude that the bilinear form $a_h$ is coercive for SIPG.
%
\paragraph{IIPG ($\theta = 0$)}
The corresponding bilinear form is written as follows, for any $\bm{v} \in \mathcal{V}_{\scaleto{DG}{4pt}}^h$:
\begin{align*}
a_h(\bm{v},\bm{v}) &= \sumOe\intOe \mathbb{C}\bm\varepsilon(\bm{v}):\bm{\varepsilon}(\bm{v}) \; dx
	- \sum_{E\in\Gamma_{iD}} \int_E \{\mathbb{C}\bm\varepsilon(\bm{v})\}:\lfloor\bm{v}\rfloor\,ds
	+ \sum_{E\in\Gamma_{iD}}\int_E \frac{k}{h_E}\mathbb{C}\lfloor\bm{v}\rfloor:\lfloor\bm{v}\rfloor\,ds.
\end{align*}
The only difference between this form and the SIPG bilinear form is the coefficient in the second term;
thus the proof of coercivity for IIPG case is identical to that for the SIPG case up to a constant.

We summarize these results.
\begin{theorem}\label{thm:coercivity}
The bilinear functional $a_h(\cdot,\cdot)$ defined in \eqref{DGweakform} is coercive if:
\begin{enumerate}[label=(\alph*)]
\item when $\theta = 1$, $k > 0$;
\item when $\theta \in \{0,-1\}$,
\[
k \geq \dfrac{m}{2} + \dfrac{C}{1-m},
\]
where $C$ is a positive constant to be calculated, and $0 < m < 1$.
\end{enumerate}
\end{theorem}
%
\subsection{Error bound}
As shown in \cite{Wihler}, one has uniform ($\lambda$-independent) convergence for the isotropic problem when linear triangles are used.
We present here a corresponding bound for transversely isotropic materials, assuming a constant fibre direction $\bm{a}$.
To establish  the bound, we adopt the same approach as in \cite{Wihler}; that is, splitting the error using a linear interpolant $\Pi_e\bm{u} \in \big[\mathcal{P}_1(\Omega_e)\big]^d$, for $\bm{u} \in \big[\mathcal{H}^2(\Omega_e)\big]^d$, which is defined by
\be\label{interpolant}
\Pi_e\bm{u}(\bar{\bm{x}}_E) := \dfrac{1}{h_E}\int_E \bm{u} \; ds \quad \forall\ E \in \partial\Omega_e,
\ee
where  $\bar{\bm{x}}_E$ is the midpoint of edge $E$.

The corresponding global interpolant $\Pi : \big[\mathcal{H}^2(\Omega_e)\big]^d \rightarrow \mathcal{V}^h_{\scaleto{DG}{4pt}}$ is defined by
\[
\Pi\bm{u}_{|_{\Omega_e}} = \Pi_e\bm{u} \quad \forall\;\Omega_e \in \mathcal{T}_h.
\]
\begin{proposition}
The interpolant has the following properties:
\begin{subequations}\label{interpolant_properties}
\begin{align}
& \int_E (\bm{u} - \Pi_e\bm{u}) \; ds = \bm{0},\label{interpolant_0}\\
& \int_E (\bm{u} - \Pi_e\bm{u}) \cdot \bm{n}_e \; ds = 0,\label{interpolant_1}\\
& \int_{\Omega_e} \nabla \cdot (\bm{u} - \Pi_e\bm{u}) \; dx = 0,\label{interpolant_2}\\
& \int_{\Omega_e} \bm{M}:\bm{\varepsilon} (\bm{u} - \Pi_e\bm{u}) \; dx = 0.\label{interpolant_3}
\end{align}
\end{subequations}
\end{proposition}
\begin{proof}
The proofs of \eqref{interpolant_0}, \eqref{interpolant_1} and \eqref{interpolant_2} are given in \cite{Wihler}.

For \eqref{interpolant_3}, using integration by parts we have
\[
\int_{\Omega_e} \bm{M}:\bm{\epsilon} (\bm{u} - \Pi_e\bm{u}) \; dx = \int_{\partial\Omega_e} (\bm{u} - \Pi_e\bm{u}) \otimes \bm{n}_e:\bm{M} \; ds - \int_{\Omega_e} (\nabla\cdot\bm{M})\cdot(\bm{u} - \Pi_e\bm{u}) \; dx = 0.
\]
\end{proof}
\begin{proposition}
The following interpolation error estimates hold:
\begin{subequations}\label{interpolant_estimates}
\begin{align}
& \|\bm{u} - \Pi_e\bm{u}\|_{0,\Omega_e} + h_e|\bm{u} - \Pi_e\bm{u}|_{1,\Omega_e} \leq Ch^2_e |\bm{u}|_{2,\Omega_e},\label{estimate_0}\\
& |\bm{u} - \Pi_e\bm{u}|_{2,\Omega_e} = |\bm{u}|_{2,\Omega_e},\label{estimate_1}\\
& \|\nabla \cdot (\bm{u} - \Pi_e\bm{u})\|_{0,\Omega_e} \leq Ch_e |\nabla \cdot \bm{u}|_{1,\Omega_e},\label{estimate_2}\\
& |\nabla \cdot (\bm{u} - \Pi_e\bm{u})|_{1,\Omega_e} = |\nabla \cdot \bm{u}|_{1,\Omega_e},\label{estimate_3}\\
& \|\bm{M}:\bm\varepsilon(\bm{u} - \Pi_e\bm{u})\|_{0,\Omega_e} \leq C h_e |\bm{M}:\bm\varepsilon(\bm{u})|_{1,\Omega_e},\label{estimate_4}\\
& |\bm{M}:\bm\varepsilon(\bm{u} - \Pi_e\bm{u})|_{1,\Omega_e} = |\bm{M}:\bm\varepsilon(\bm{u})|_{1,\Omega_e}\label{estimate_5}
\end{align}
\end{subequations}
where $C>0$ is in each case a constant independent of $h_e$ and $\bm{u}$.
\end{proposition}
\begin{proof}
The proofs for \eqref{estimate_0}-\eqref{estimate_3} are given in \cite{Wihler}.

Proof of \eqref{estimate_5}.\ \ Setting $\bm{U} = \bm{u} - \Pi_e\bm{u}$, then since $\Pi_e\bm{u} \in \big[\mathcal{P}_1(\Omega_e)\big]^d$, it follows that
\be\label{M_epsilon_H1}
|\bm{M}:\bm\varepsilon(\bm{U})|_{1,\Omega_e} = |\bm{M}:\bm\varepsilon(\bm{u})|_{1,\Omega_e}.
\ee
Proof of \eqref{estimate_4}. Lemma A.3 in \cite{Wihler} applied to $\bm{M}:\bm\varepsilon(\bm{U})$ gives
\[
\left\|\bm{M}:\bm\varepsilon(\bm{U}) - \dfrac{1}{|\Omega_e|}\int_{\Omega_e}\bm{M}:\bm\varepsilon(\bm{U})\;dx\right\|_{0,\Omega_e} \leq Ch_e|\bm{M}:\bm\varepsilon(\bm{U})|_{1,\Omega_e}.
\]
With \eqref{interpolant_3} and \eqref{M_epsilon_H1}, we obtain \eqref{estimate_4}.
\end{proof}
Let $\bm{u} \in \big[\mathcal{H}^2(\Omega_e)\big]^d$ be the exact solution to the problem and $\bm{u}_h \in \mathcal{V}^h_{DG}$ the corresponding finite element approximation;
then the approximation error is
\begin{align*}
\bm{e} &= \bm{u} - \bm{u}_h\\
	&= \underbrace{\bm{u} - \Pi\bm{u}}_{\bm{\eta}} + \underbrace{\Pi\bm{u} - \bm{u}_h}_{\bm{\xi}}.
\end{align*}
In particular $\bm{\xi}_{|_{\Omega_e}} \in [\mathcal{P}_1(\Omega_e)]^d$.
The DG-norm of the error is
\[
||\bm{e}||^2_{DG} \leq ||\bm{\eta}||^2_{DG} + ||\bm{\xi}||^2_{DG}.
\]
Starting with $||\bm{\eta}||^2_{DG}$ we have, from \eqref{DG_norm} and using \eqref{bound_a} and \eqref{bound_b},
\begin{align}
||\bm{\eta}||^2_{DG} & = \sumOe \|\bm{\varepsilon}(\bm{\eta})\|_{0,\Omega_e}^2 + \dfrac{1}{2}\sumGiD \dfrac{1}{h_E}\|\lfloor\bm{\eta}\rfloor\|_{0,E}^2 \nonumber\\
 &\leq C \Big( \sumOe \|\nabla\bm{\eta}\|^2_{0,\Omega_e}  +  \sumOe \sum_{E \in \partial\Omega_e^{iD}} \dfrac{1}{h_E}\|\bm{\eta}\|^2_{0,E} \Big) \nonumber \\
& \leq C\Big( \sumOe |\bm{\eta}|^2_{1,\Omega_e}  +  \sumOe \left( h_e^{-2}\|\bm{\eta}\|^2_{0,\Omega_e} + |\bm\eta|^2_{1,\Omega_e} \right)\Big) \nonumber \\
& \leq C\sumOe h_e^2 |\bm{u}|^2_{2,\Omega_e}  \label{eta_DG_norm}\,.
\end{align}
To bound $||\bm{\xi}||^2_{DG}$,
we have $||\bm{\xi}||^2_{DG} \leq \frac{1}{K}|a_h(\bm{\eta},\bm{\xi})|$ from coercivity of the bilinear form.
To bound $|a_h(\bm{\eta},\bm\xi)|$, the technique is to extract a factor of $\DGnorm{\bm\xi}$, leaving some function in terms of $\bm\eta$ which will be bounded by norms of the exact solution $\bm{u}$ from each term.
The fact that $\bm{\xi} \in [\mathcal{P}_1(\Omega)]^d$ so that $\bm{\varepsilon}(\bm{\xi}), \nabla\cdot\bm{\xi}$, and $\nabla\bm{\xi}$ are constants, is also useful.

We have
\[
|a_h(\bm{\eta},\bm{\xi})| \leq |a^{iso}_h(\bm{\eta},\bm{\xi})| + |a^{ti}_h(\bm{\eta},\bm{\xi})|,
\]
where the isotropic part is bounded as follows (see \cite{Beverley}):
\be\label{iso_error_bound}
|a^{iso}_h(\bm{\eta},\bm{\xi})| \leq C ||\bm{\xi}||_{DG} \left[\sumOe h^2_e \left(\mu_t^2|\bm{u}|^2_{2,\Omega_e} + \lambda^2|\nabla\cdot\bm{u}|^2_{1,\Omega_e}\right)\right]^{1/2},
\ee
For the remaining part, we have
\begin{align}\label{a_eta_xi}
&|a^{ti}_h(\bm{\eta},\bm{\xi})| \leq\nonumber\\
&+ \alpha\cancel{\Big|\sumintOe{ (\bm{M}:\bm{\varepsilon}(\bm{\eta}))\nabla\cdot\bm{\xi} + \nabla\cdot\bm{\eta}(\bm{M}:\bm{\varepsilon}(\bm{\xi}))}\Big|}
	+ \beta \cancel{\Big|\sumintOe{(\bm{M}:\bm{\varepsilon}(\bm{\eta}))(\bm{M}:\bm{\varepsilon}(\bm{\xi}))}\Big|}\nonumber\\
	&+ \underbrace{2\gamma \Big|\sumintOe{ \bm{\varepsilon}(\bm{\eta}) \bm{M}:\bm{\varepsilon}(\bm{\xi})}\Big|}_{I}
	+ \cancel{\Big|\sumGiD \theta \int_E \lfloor \bm{\eta} \rfloor : \{\bm{\sigma}^{ti}(\bm{\xi})\}\ ds\Big|}\nonumber\\
	&+ \underbrace{\beta \Big|\sumGiD \int_E \lfloor \bm{\xi} \rfloor:\{\bm{M}(\bm{M}:\bm{\varepsilon}(\bm{\eta}))\}\ ds\Big|}_{II}
	+ \underbrace{\alpha \Big|\sumGiD \int_E \lfloor \bm{\xi} \rfloor:\{(\bm{M}:\bm{\varepsilon}(\bm{\eta}))\mathbf{I}+(\nabla\cdot\bm{\eta})\bm{M}\}\ ds\Big|}_{III}\nonumber\\
	&+ \underbrace{\gamma \Big|\sumGiD \int_E \lfloor \bm{\xi} \rfloor:\{\bm{\varepsilon}(\bm{\eta}) \bm{M}+\bm{ M\varepsilon}(\bm{\eta})\}\ ds\Big|}_{IV}
	+\underbrace{k\beta \Big| \sumGiD \dfrac{1}{h_E} \int_E (\bm{M}:\lfloor\bm{\eta}\rfloor) (\bm{M}:\lfloor\bm{\xi}\rfloor)\ ds\Big|}_{V}\nonumber\\
	&+\underbrace{k\alpha \Big| \sumGiD \dfrac{1}{h_E} \int_E[\bm{\eta}](\bm{M}:\lfloor\bm{\xi}\rfloor) + [\bm{\xi}](\bm{M}:\lfloor\bm{\eta}\rfloor)\ ds\Big|}_{VI}
	+\underbrace{k\gamma \Big| \sumGiD \dfrac{1}{h_E} \int_E\lfloor\bm{\xi}\rfloor:(\lfloor\bm{\eta}\rfloor\bm{M}+\bm{M}\lfloor\bm{\eta}\rfloor)\ ds\Big|}_{VII}.
\end{align}
The struck-through terms are zero from the properties of the interpolant.
We now bound each of the remaining terms in \eqref{a_eta_xi}.
\begin{align*}
I &= 2\gamma \Big|\sumintOe{ \bm{\varepsilon}(\bm{\eta}) \bm{M}:\bm{\varepsilon}(\bm{\xi})}\Big|\\
& \leq 2\gamma \left(\sumOe \|\bm{\varepsilon}(\bm{\eta}) \bm{M}\|^2_{0,\Omega_e}\right)^{1/2} \left(\sumOe \|\bm{\varepsilon}(\bm{\xi})\|^2_{0,\Omega_e}\right)^{1/2}\\
& \leq C\gamma \left(\sumOe \|\bm{\varepsilon}(\bm{\eta})\|^2_{0,\Omega_e}\right)^{1/2} \left(\sumOe \|\bm{\varepsilon}(\bm{\xi})\|^2_{0,\Omega_e}\right)^{1/2}\\
& \leq C \gamma \|\bm\xi\|_{DG} \left(\sumOe \|\nabla\bm{\eta}\|^2_{0,\Omega_e}\right)^{1/2}\\
& \leq C \gamma\|\bm\xi\|_{DG} \left(\sumOe |\bm{\eta}|^2_{1,\Omega_e}\right)^{1/2}\\
& \leq C \gamma\|\bm\xi\|_{DG} \left(\sumOe h^2_e |\bm{u}|^2_{2,\Omega_e}\right)^{1/2} \qquad{\color{gray}(\text{using }\eqref{estimate_0})}.
\end{align*}
\begin{align*}
II &= \beta \Big|\sumGiD \int_E \lfloor \bm{\xi} \rfloor:\{\bm{M}(\bm{M}:\bm{\varepsilon}(\bm{\eta}))\}\ ds\Big|\\
& \leq \beta \left(\sumGiD \dfrac{1}{h_E}\|\lfloor \bm{\xi} \rfloor\|^2_{0,E}\right)^{1/2} \left(\sumGiD h_E\|\{\bm{M}(\bm{M}:\bm{\varepsilon}(\bm{\eta}))\}\|^2_{0,E}\right)^{1/2}\\
& \leq \beta \|\bm\xi\|_{DG} \left(\sumGiD h_E\|\{\bm{M}:\bm{\varepsilon}(\bm{\eta})\}\|^2_{0,E}\right)^{1/2}\\
& \leq C \beta \|\bm\xi\|_{DG} \left(\sumOe h_e \|\bm{M}:\bm{\varepsilon}(\bm{\eta})\|^2_{0,\partial\Omega_e^{iD}}\right)^{1/2}\quad{\color{gray}(\text{using }\eqref{bound_b})}\\
& \leq C \beta \|\bm\xi\|_{DG} \left(\sumOe \|\bm{M}:\bm{\varepsilon}(\bm{\eta})\|^2_{0,\Omega_e} + h_e^2 |\bm{M}:\bm{\varepsilon}(\bm{\eta})|_{1,\Omega_e}\right)^{1/2} {\color{gray}(\text{using }\eqref{bound_a}}\\
& \leq C \beta \|\bm\xi\|_{DG} \left(\sumOe h^2_e |\bm{M}:\bm{\varepsilon}(\bm{u})|^2_{1,\Omega_e}\right)^{1/2} \qquad{\color{gray}(\text{using }\eqref{estimate_4} \text{ and } \eqref{estimate_5})}.
\end{align*}
\begin{align*}
III_a := \alpha \Big|\sumGiD \int_E \lfloor \bm{\xi} \rfloor:\{(\bm{M}:\bm{\varepsilon}(\bm{\eta}))\mathbf{I}\}\ ds\Big|
& \leq C \alpha \|\bm\xi\|_{DG} \left(\sumOe h^2_e |\bm{M}:\bm{\varepsilon}(\bm{u})|^2_{1,\Omega_e}\right)^{1/2} \qquad{\color{gray}(\text{similar to } II)}. 
\\
\end{align*}
\begin{align*}
III_b &:= \alpha \Big|\sumGiD \int_E \lfloor \bm{\xi} \rfloor:\{(\nabla\cdot\bm{\eta})\bm{M}\}\ ds\Big|\\
& \leq \alpha \left(\sumGiD \dfrac{1}{h_E}\|\lfloor \bm{\xi} \rfloor\|^2_{0,E}\right)^{1/2} \left(\sumGiD h_E\|\{(\nabla\cdot\bm{\eta})\bm{M}\}\|^2_{0,E}\right)^{1/2}\\
& \leq C \alpha \|\bm\xi\|_{DG} \left(\sumOe h_e \|\nabla\cdot\bm{\eta}\|^2_{0,\partial\Omega_e^{iD}}\right)^{1/2} \qquad{\color{gray}(\text{using }\eqref{bound_b})}\\
& \leq C \alpha \|\bm\xi\|_{DG} \left(\sumOe \|\nabla\cdot\bm{\eta}\|^2_{0,\Omega_e} + h_e^2 |\nabla\cdot\bm{\eta}|^2_{1,\Omega_e}\right)^{1/2} \qquad{\color{gray}(\text{using }\eqref{bound_a})}\\
& \leq C \alpha \|\bm\xi\|_{DG} \left(\sumOe h_e^2|\nabla\cdot\bm{u}|^2_{1,\Omega_e}\right)^{1/2}  \qquad{\color{gray}(\text{using }\eqref{estimate_2} \text{ and }\eqref{estimate_3})}.
\end{align*}
\begin{align*}
IV &= \gamma \Big|\sumGiD \int_E \lfloor \bm{\xi} \rfloor:\{\bm{\varepsilon}(\bm{\eta}) \bm{M}\}\ ds\Big|\\
& \leq \gamma \left(\sumGiD \dfrac{1}{h_E}\|\lfloor \bm{\xi} \rfloor\|^2_{0,E}\right)^{1/2} \left(\sumGiD h_E\|\{\bm{\varepsilon}(\bm{\eta}) \bm{M}\}\|^2_{0,E}\right)^{1/2}\\
& \leq C \gamma\|\bm\xi\|_{DG} \left(\sumOe h^2_e |\bm{u}|^2_{2,\Omega_e}\right)^{1/2} ,
\end{align*}
following steps similar to those in $I$. 
\begin{align*}
V &= k\beta \Big| \sumGiD \dfrac{1}{h_E} \int_E (\bm{M}:\lfloor\bm{\eta}\rfloor) (\bm{M}:\lfloor\bm{\xi}\rfloor)\ ds\Big|\\
& \leq k\beta \left(\sumGiD \dfrac{1}{h_E} \|\bm{M}:\lfloor\bm{\xi}\rfloor\|^2_{0,E}\right)^{1/2} \left(\sumGiD \dfrac{1}{h_E} \|\bm{M}:\lfloor\bm{\eta}\rfloor\|^2_{0,E}\right)^{1/2}\\
& \leq C \beta \|\bm\xi\|_{DG} \left(\sumGiD \dfrac{1}{h_E} \|\lfloor\bm{\eta}\rfloor\|^2_{0,E}\right)^{1/2}\\
& \leq C \beta \|\bm\xi\|_{DG} \left(\sumOe \sum_{E\in\partial\Omega_e^{iD}} \dfrac{2}{h_E} \|\bm{\eta}\|_{0,E}\right)^{1/2} \qquad{\color{gray}(\text{using }\eqref{bound_d})}\\
& \leq C \beta \|\bm\xi\|_{DG} \left(\sumOe h_e^{-2} \|\bm{\eta}\|_{0,\Omega_e} + |\bm{\eta}|_{1,\Omega_e}\right)^{1/2} \qquad{\color{gray}(\text{using }\eqref{bound_a})}\\
& \leq C \beta \|\bm\xi\|_{DG} \left(\sumOe h^2_e |\bm{u}|^2_{2,\Omega_e}\right)^{1/2} \qquad{\color{gray}(\text{using }\eqref{estimate_0})}.
\end{align*}
\begin{align*}
VI_a := k\alpha \Big| \sumGiD \dfrac{1}{h_E} \int_E[\bm{\eta}](\bm{M}:\lfloor\bm{\xi}\rfloor)\ ds \Big|
& \leq C \alpha \|\bm\xi\|_{DG} \left(\sumOe h^2_e |\bm{u}|^2_{2,\Omega_e}\right)^{1/2} \qquad{\color{gray}(\text{similar to } V)}.
\end{align*}
\begin{align*}
VI_b := k\alpha \Big| \sumGiD \dfrac{1}{h_E} \int_E[\bm{\xi}](\bm{M}:\lfloor\bm{\eta}\rfloor)\ ds \Big|
& \leq C \alpha \|\bm\xi\|_{DG} \left(\sumOe h^2_e |\bm{u}|^2_{2,\Omega_e}\right)^{1/2} \qquad{\color{gray}(\text{similar to } V)}.
\end{align*}
\begin{align*}
VII = k \gamma \Big| \sumGiD \dfrac{1}{h_E} \int_E \lfloor\bm{\xi}\rfloor : \lfloor\bm{\eta}\rfloor \bm{M}\ ds \Big|
& \leq C \alpha \|\bm\xi\|_{DG} \left(\sumOe h^2_e |\bm{u}|^2_{2,\Omega_e}\right)^{1/2} \qquad{\color{gray}(\text{similar to } V)}.
\end{align*}
We use these results to bound each term of $|a^{ti}_h(\bm{\eta},\bm{\xi})|$, which leads to
\[
|a^{ti}_h(\bm{\eta},\bm{\xi})| \leq C ||\bm{\xi}||_{DG} \left(\sumOe h^2_e \Big((\alpha^2+\beta^2+\gamma^2)|\bm{u}|^2_{2,\Omega_e} + \alpha^2|\nabla\cdot\bm{u}|^2_{1,\Omega_e} + (\alpha^2+\beta^2) |\bm{M}:\bm\varepsilon(\bm{u})|^2_{1,\Omega_e}\Big)\right)^{1/2}.
\]
Therefore,
\begin{small}
\be\label{xi_DG_norm}
||\bm{\xi}||^2_{DG} \leq \dfrac{C}{K^2} \sumOe h^2_e \Big((\mu_t^2+\alpha^2+\beta^2+\gamma^2)|\bm{u}|^2_{2,\Omega_e} + (\lambda^2+\alpha^2)|\nabla\cdot\bm{u}|^2_{1,\Omega_e} + (\alpha^2+\beta^2) |\bm{M}:\bm\varepsilon(\bm{u})|^2_{1,\Omega_e} \Big),
\ee
\end{small}
where $K$ is the coercivity constant defined by \eqref{NIPG_coercivity_bound} for the NIPG case, and by \eqref{SIPG_coercivity_bound} for the SIPG and IIPG cases.

With \eqref{eta_DG_norm} and \eqref{xi_DG_norm}, the full DG error bound is
\begin{align}
||\bm{e}||^2_{DG}
	&\leq \dfrac{C}{K^2} h^2 \Big((\mu_t^2+\alpha^2+\beta^2+\gamma^2)|\bm{u}|^2_{2,\Omega} + (\lambda^2+\alpha^2)|\nabla\cdot\bm{u}|^2_{1,\Omega} + (\alpha^2+\beta^2) |\bm{M}:\bm\varepsilon(\bm{u})|^2_{1,\Omega}\Big).\label{DGerror}
\end{align}
{\bf Remark.} For the case of isotropy, the error estimate is (see \cite{Beverley})
\be\label{isoerror}
||\bm{e}||^2_{DG} \leq \dfrac{C}{K^2} h^2 \left(\mu_t^2|\bm{u}|^2_{2,\Omega} + \lambda^2|\nabla\cdot\bm{u}|^2_{1,\Omega}\right).
\ee
An a priori estimate due to Brenner \& Sung \cite{Brenner1992} for the case of problems on polygonal domains $\Omega \subset \mathbb{R}^2$ allows the right-hand side of \eqref{isoerror} to be bounded independent of $\lambda$, thus confirming the locking-free behaviour of the DG formulation in the incompressible limit.
A similar estimate for the transversely isotropic problem is not available; however, one would expect that an analogous estimate would allow terms of the form
\be\label{TI_estimate}
(\mu_t^2+\alpha^2+\gamma^2)|\bm{u}|^2_{2,\Omega} + (\lambda^2+\alpha^2)|\nabla\cdot\bm{u}|^2_{1,\Omega} + (\alpha^2+\beta^2) |\bm{M}:\bm\varepsilon(\bm{u})|^2_{1,\Omega}
\ee
to be bounded independent of $\lambda$ and $\beta$.
The presence of $\beta$ in the first term of \eqref{DGerror} suggests that locking may occur in the inextensible limit.
Numerical experiments discussed in Section \ref{Sec:Numerical} will explore these features.

The term that leads to the undesirable $\beta$-dependence in the error bound is term $V$ in \eqref{a_eta_xi}.
To circumvent the $\beta$-dependence, one would need to find a way to modify the formulation in such a way that this term is eliminated.
\section{Under-integration}\label{Sec:UI}
In the context of lowest-order approximations on quadrilaterals (in two dimensions) or hexahedra (in three), the use of under-integration in the numerical implementations is generally equivalent to projecting the integrand in a formulation onto the space of constants (see for example \cite{Arnold1981}). 
The latter is in turn equivalent to a mixed formulation of the problem.
For the case of isotropy, using bilinear elements, the undesirable $\lambda$-dependency of the error bound in the incompressible limit may be circumvented by under-integrating the problematic terms \cite{Beverley}.
The same approach is used here in order to overcome locking in the inextensible limit: the $\beta$-stabilization term $V$ will be under-integrated.

It should be noted that the equivalence between under-integration and a mixed formulation does not carry over to certain situations. 
For example, for problems posed in coordinate systems other than cartesian, in which the volume element $dV := dx\, dy\, dz$ becomes, for cylindrical coordinates, $dV = r\,dr\,d \theta\, dz$, the radial coordinate becomes an obstacle to showing such equivalence. 
Such situations require alternative formulations and analyses of both displacement-based and mixed formulations (cf. \cite{Hughes1987}, Section 4.5), and are not pursued here.

We will adopt a form of the formulation in which the bilinear form is given by \eqref{ah_DG_k} with $k_\mu$ replaced by $2k + k_\mu$ and all other stabilization parameters equal to $k$, assuming $k, k_\mu>0$.

The resulting bilinear form is then
\be\label{ah_DG_new}
\overline{a}_h(\bm{u},\bm{v})
:= a_h(\bm{u},\bm{v}) + \mu_t \sum_{E\in\Gamma_{iD}} \dfrac{k_\mu}{h_E}\int_E \lfloor \bm{u} \rfloor:\lfloor \bm{v} \rfloor\ ds.
\ee
Note that this bilinear form is coercive, as is easily established using Theorem \ref{thm:coercivity}.

If we define by $\Pi_0$ the $\mathcal{L}^2$-orthogonal projection onto the space of constants,
the new DG formulation with under-integration is:
\[
\overline{a}_h^{\scaleto{UI}{4pt}}(\bm{u},\bm{v}) = {l}_h^{\scaleto{UI}{4pt}}(\bm{v})
\]
where
\be\label{ahUI}
\overline{a}_h^{\scaleto{UI}{4pt}}(\bm{u},\bm{v}) = \overline{a}_h(\bm{u},\bm{v})
+ k \beta\sumintGiD{(\bm{M}:\lfloor\bm{u}\rfloor)(\Pi_0-\bm{I})(\bm{M}:\lfloor\bm{v}\rfloor)},
\ee
and
\be\label{lhUI}
{l}_h^{\scaleto{UI}{4pt}}(\bm{v}) = {l}_h(\bm{v})
+ k \beta\sumintG{D}{(\bm{M}:\bm{g}\otimes\bm{n})(\Pi_0-\bm{I})(\bm{M}:\bm{v}\otimes\bm{n})}.
\ee
Note that under-integration of the edge term
$
\int_E \lfloor\cdot\rfloor:\{\bm\sigma(\cdot)\} \;ds
$
is not necessary, since for any $\bm{u} \in [\mathcal{P}_1(\Omega_e)]^d$ the integrand is linear, so that one-point integration is exact.
The integrand in the term $V$ in \eqref{a_eta_xi} is then replaced with its projected quantity, and is easily shown to be zero.
However, this has involved modification of the DG formulation itself, so that it is necessary to show coercivity and consistency of the modified bilinear form.
\subsection{Coercivity}
Each IP method will be investigated separately.
\paragraph{NIPG $(\theta = 1)$}
We have
\begin{align}
\overline{a}_h^{\scaleto{UI}{4pt}}(\bm{v},\bm{v})
=& \sumintOe{\mathbb{C}\bm\varepsilon(\bm{v}):\bm\varepsilon(\bm{v})}
	+ \sum_{E\in\Gamma_{iD}} \dfrac{k}{h_E}\int_E\mathbb{C} \lfloor\bm{v}\rfloor : \lfloor\bm{v}\rfloor\ ds
	+ \mu_t \sum_{E\in\Gamma_{iD}} \dfrac{k_\mu}{h_E}\int_E |\lfloor \bm{v} \rfloor|^2\ ds\nonumber\\
	&+ \beta \sum_{E\in\Gamma_{iD}} \dfrac{k}{h_E}\int_E (\Pi_0-\bm{I}) \big(\bm{M}:\lfloor\bm{v}\rfloor\big) \big(\bm{M}:\lfloor\bm{v}\rfloor\big)\ ds\nonumber\\
=& {a}_h(\bm{v},\bm{v})
	+ \mu_t \sum_{E\in\Gamma_{iD}} \dfrac{k_\mu}{h_E}\int_E |\lfloor \bm{v} \rfloor|^2\ ds
	+ \beta \sum_{E\in\Gamma_{iD}} \dfrac{k}{h_E}\int_E (\Pi_0-\bm{I}) \big(\bm{M}:\lfloor\bm{v}\rfloor\big) \big(\bm{M}:\lfloor\bm{v}\rfloor\big)\ ds,\label{ah_UI_vec}
\end{align}
noting that ${a}_h(\bm{v},\bm{v})$ is defined as in \eqref{NIPG_coerc}.

We define
\be\label{B}
\mathscr{B} := \mu_t \sum_{E\in\Gamma_{iD}} \dfrac{k_\mu}{h_E}\int_E|\lfloor \bm{v} \rfloor|^2\ ds
	+ \beta \sum_{E\in\Gamma_{iD}} \dfrac{k}{h_E}\int_E (\Pi_0-\bm{I}) \big(\bm{M}:\lfloor\bm{v}\rfloor\big) \big(\bm{M}:\lfloor\bm{v}\rfloor\big)\ ds.
\ee
For ease, we set $\bm{m} = \bm{v}_i - \bm{v}_e$ and denote by $\bm{n}$ the outward unit normal vector, giving
\begin{align*}
|\lfloor \bm{v} \rfloor|^2 &= |\bm{m}\otimes \bm{n}|^2 = \bm{m}\cdot\bm{m},\\
\bm{M}:\lfloor \bm{v} \rfloor &= (\bm{a}\otimes\bm{a}):(\bm{m}\otimes \bm{n}) = (\bm{a}\cdot\bm{m})(\bm{a}\cdot\bm{n}).
\end{align*}
Then
\be\label{B_vec}
\mathscr{B} = \mu_t \sum_{E\in\Gamma_{iD}} \dfrac{k_\mu}{h_E}\int_E \bm{m}\cdot\bm{m}\ ds
	+ \beta \sum_{E\in\Gamma_{iD}} \dfrac{k}{h_E}\int_E \Big((\Pi_0-\bm{I}) (\bm{a}\cdot\bm{m})\Big)^2 (\bm{a}\cdot\bm{n})^2\ ds.
\ee
We have
\begin{align*}
\bm{m}\cdot\bm{m} &\geq (\bm{a}\cdot\bm{m})^2\\
&\geq (\bm{a}\cdot\bm{m})^2 (\bm{a}\cdot\bm{n})^2 \qquad{\color{gray}(\text{since }(\bm{a}\cdot\bm{n})^2\leq 1)}\\
&= (\bm{a}\cdot\bm{n})^2 \left(\dfrac{1}{2}(\bm{a}\cdot\bm{m})^2 + \dfrac{1}{2}(\bm{a}\cdot\bm{m})^2\right).
\end{align*}
Noting that
\[
\int_E\Big(\Pi_0(\bullet)\Big)^2 \,ds\leq \int_E\bullet^2\,ds,
\]
we then have
\be\label{bound_vec}
\int_E\bm{m}\cdot\bm{m} \,ds \geq \int_E(\bm{a}\cdot\bm{n})^2 \left(\dfrac{1}{2}(\bm{a}\cdot\bm{m})^2 + \dfrac{1}{2}\Big(\Pi_0(\bm{a}\cdot\bm{m})\Big)^2\right) \,ds.
\ee
Returning to \eqref{B_vec}, using \eqref{bound_vec} we obtain
\[
\mathscr{B} \geq \sum_{E\in\Gamma_{iD}} \dfrac{1}{h_E}\int_E (\bm{a}\cdot\bm{n})^2
\left[ \left(\dfrac{k_\mu\mu_t}{2} - k\beta\right) (\bm{a}\cdot\bm{m})^2 + \left(\dfrac{k_\mu\mu_t}{2} + k\beta\right) \Big(\Pi_0(\bm{a}\cdot\bm{m})\Big)^2\right] \,ds.
\]
The term on the right-hand-side is non-negative if
\be\label{k_choice}
\begin{dcases}
\dfrac{k_\mu\mu}{2} - k\beta \geq 0,\\
\dfrac{k_\mu\mu}{2} + k\beta \geq 0,
\end{dcases}
\quad \Leftrightarrow \quad
\dfrac{2k|\beta|}{\mu_t} \leq k_\mu.
\ee
Returning to \eqref{ah_UI_vec}, by choosing $k$ and $k_\mu$ as in \eqref{k_choice}, we have
\[
\overline{a}_h^{\scaleto{UI}{4pt}}(\bm{v},\bm{v})
\geq {a}_h(\bm{v},\bm{v}).
\]
From \eqref{ah_NIPG}, we have
\[
\overline{a}_h^{\scaleto{UI}{4pt}}(\bm{v},\bm{v}) \geq K \DGnorm{\bm{v}}^2,
\]
with
\be\label{NIPG_UI_coercivity_bound}
K = \Lambda_{min} \min \big\{1, 2k\big\}.
\ee
Thus, the under-integrated NIPG formulation is coercive.
%
\paragraph{SIPG $(\theta = -1)$}
We have
\begin{align*}
\overline{a}_h^{\scaleto{UI}{4pt}}(\bm{v},\bm{v})
=& \sumintOe{\mathbb{C}\bm\varepsilon(\bm{v}):\bm\varepsilon(\bm{v})}
	-2 \sum_{E\in \Gamma_{iD}}\int_E \lfloor\bm{v}\rfloor:\{\mathbb{C}\bm\varepsilon(\bm{v})\}\;ds 
	+ \sum_{E\in\Gamma_{iD}} \dfrac{k}{h_E}\int_E\mathbb{C} \lfloor\bm{v}\rfloor : \lfloor\bm{v}\rfloor\ ds\\
	&+ \mu_t \sum_{E\in\Gamma_{iD}} \dfrac{k_\mu}{h_E}\int_E |\lfloor \bm{v} \rfloor|^2\ ds
	+ \beta \sum_{E\in\Gamma_{iD}} \dfrac{k}{h_E}\int_E (\Pi_0-\bm{I}) \big(\bm{M}:\lfloor\bm{v}\rfloor\big) \big(\bm{M}:\lfloor\bm{v}\rfloor\big)\ ds\\
	=& a_h(\bm{v},\bm{v}) + \mathscr{B},
\end{align*}
note that ${a}_h(\bm{v},\bm{v})$ is defined as in \eqref{SIPG_coerc}, and $\mathscr{B}$ as in \eqref{B}.

From the NIPG coercivity proof above, for $k$ and $k_\mu$ that satisfy \eqref{k_choice}, we have
\[
\overline{a}_h^{\scaleto{UI}{4pt}}(\bm{v},\bm{v})
\geq {a}_h(\bm{v},\bm{v}).
\]
From \eqref{ah_SIPG}, we have:
\[
\overline{a}_h^{\scaleto{UI}{4pt}}(\bm{v},\bm{v}) \geq K \DGnorm{\bm{v}}^2,
\]
with
\be\label{SIPG_UI_coercivity_bound}
K = m\Lambda_{min},
\ee
such that $0<m<1$ and $k \geq \frac{m}{2} + \frac{C}{1-m}$, for a positive constant $C$ to be determined.
%
\paragraph{IIPG $(\theta = 0)$}
The proof of coercivity for the case of IIPG with under-integration is identical to that for the case of SIPG with under-integration up to a constant.
\begin{theorem}
The bilinear functional $\overline{a}_h^{\scaleto{UI}{4pt}}(\cdot,\cdot)$ defined in \eqref{ahUI} is coercive if, for $k,k_\mu>0$
\begin{enumerate}[label=(\alph*)]
\item when $\theta = 1$,
\[
\dfrac{2k|\beta|}{\mu_t} \leq k_\mu;
\]
\item when $\theta \in \{-1,0\}$,
\[
\dfrac{2k|\beta|}{\mu_t} \leq k_\mu, \quad \text{and} \quad k\geq\dfrac{m}{2} + \dfrac{C}{1-m},
\]
where $C$ is a positive constant to be calculated, and $0<m<1$.
\end{enumerate}
\end{theorem}

\subsection{Consistency}
With the continuous exact solution $\bm{u} \in [\mathcal{H}^2(\Omega)]^d$ satisfying the properties given in Section \ref{sec:consistency}, we have
\begin{align*}
\bar{a}_h^{\scaleto{UI}{4pt}}(\bm{u},\bm{v}) - l_h^{\scaleto{UI}{4pt}}(\bm{v})
=& \sumintOe{\bm\sigma(\bm{u}):\bm\varepsilon(\bm{v})}
- \sumGiD \int_E \lfloor\bm{v}\rfloor:\{\bm\sigma(\bm{u})\} \;ds
- \sumintOe{\bm{f}\cdot\bm{v}}\\
&- \sum_{E\in\Gamma_N} \int_E \bm{h}\cdot\bm{v} \;ds.
\end{align*}
Since $\bm{u}$ satisfies the weak form,
\begin{align*}
\sumintOe{\bm\sigma(\bm{u}):\bm\varepsilon(\bm{v})}
&= \sumintOe{\bm{f}\cdot\bm{v}} + \sum_{\Omega_e\in\mathcal{T}_h} \int_{\partial\Omega_e} \bm{v}\otimes\bm{n} : \bm\sigma(\bm{v})\\
&= \sumintOe{\bm{f}\cdot\bm{v}} + \sum_{E\in \Gamma}\int_E \lfloor\bm{v}\rfloor:\{\bm\sigma(\bm{u})\} + \sum_{E\in\Gamma_{int}} \int_E \lfloor\bm\sigma(\bm{u})\rfloor\cdot\{\bm{v}\}\\
&= \sumintOe{\bm{f}\cdot\bm{v}} + \sum_{E\in \Gamma_{iD}}\int_E \lfloor\bm{v}\rfloor:\{\bm\sigma(\bm{u})\} +  \sum_{E\in\Gamma_N} \int_E \bm{t}\cdot\bm{v} \;ds.
\end{align*}
Therefore,
\[
\bar{a}_h^{\scaleto{UI}{4pt}}(\bm{u},\bm{v}) - l_h^{\scaleto{UI}{4pt}}(\bm{v}) = 0
\]
as desired.

\paragraph{Error bound}
The approximation error is now bounded 
\be\label{DGUI_error}
||\bm{e}||^2_{DG} \leq \dfrac{C}{K^2} h^2 \Big((\mu_t^2+\alpha^2 + \gamma^2)|\bm{u}|^2_{2,\Omega} + (\lambda^2+\alpha^2)|\nabla\cdot\bm{u}|^2_{1,\Omega} + (\alpha^2+\beta^2) |\bm{M}:\bm\varepsilon(\bm{u})|^2_{1,\Omega}\Big),
\ee
where $K$ is the coercivity constant defined by \eqref{NIPG_UI_coercivity_bound} for the NIPG case, and by \eqref{SIPG_UI_coercivity_bound} for the SIPG and IIPG cases.
The first term on the right-hand side is now independent of $\beta$.
For the case of isotropy, the following uniform estimate is proposed by Brenner \& Sung in \cite{Brenner1992}:
\be\label{B&S}
\|\bm{u}\|_{2,\Omega} + \lambda \|\nabla\cdot\bm{u}\|_{1,\Omega} \leq C \left(\|\bm{f}\|_{0,\Omega}+\|\bm{g}\|_{0,\Gamma_D}\right).
\ee
The bound \eqref{DGUI_error} is thus in a form that would be expected to lead to a uniform estimate, by analogy with the bound \eqref{B&S}.
The behaviour of the under-integrated DG formulation will be explored further in the next section.
\section{Numerical tests}\label{Sec:Numerical}
In this section, we present the results of numerical simulations of two model problems to illustrate the formulations discussed in the preceding sections.
All examples are under conditions of plane strain and based on three- and six-noded triangular elements with standard linear and quadratic interpolations of the displacement field.
We fix the value of the two Poisson's ratios to be equal, $\nu_t = \nu_l = \nu$, and also set $\mu_l = \mu_t$, as defined in \eqref{material_parameters_SC}.
For general states of stress in plane strain, incompressible behaviour occurs only for the case $p=1$ and $\nu = 1/2$.
To investigate the behaviour at near-incompressibility, we choose $\nu = 0.49995$.
For conditions \eqref{psc_2} and \eqref{psc_3} to be satisfied, we will require $p \geq 1$ for this case.
Recall that inextensible behaviour occurs for the case $p \rightarrow \infty$.
The behaviour at near-inextensibility is investigated by choosing higher values of $p \,\,(\geq 10)$.

Within the constraints of the coercivity requirements, the following values for the stabilization parameters are used for all the methods:
$k_\mu = k_\alpha = k_\gamma = 10$, and $k_\lambda = k_\beta = 100$.

Define $\hat{a} := \widehat{(Ox,\bm{a})}$, the angle between the $x$-axis and the fibre direction $\bm{a}$.
For each problem, we consider values for $\hat{a}$ in the range $0 \leq \hat{a} \leq \pi$.

Tip displacement results are shown for only one refinement level in each problem case.
Error results are shown on a sequence of meshes at increased refinement levels.

The results in the examples that follow are for the following element choices:
\begin{table}[h!]
\begin{tabular}{lp{12cm}}
$Exact$ & The analytical solution\\
$P_1\_CG$ & The standard displacement formulation of order 1\\
$P_2\_CG$ & The standard displacement formulation of order 2\\
$P_1\_NIPG$ & The nonsymmetric interior penalty method of order 1\\
$P_1\_SIPG$ & The symmetric interior penalty method of order 1\\
$P_1\_IIPG$ & The incomplete interior penalty method of order 1\\
$P_1\_NIPG\_UI_{\beta}$ & The nonsymmetric interior penalty method of order 1 with under-integration of the $\beta$- stabilization term\\
$P_1\_SIPG\_UI_{\beta}$ & The symmetric interior penalty method of order 1 with under-integration of the $\beta$- stabilization term\\
$P_1\_IIPG\_UI_{\beta}$ & The incomplete interior penalty method of order 1 with under-integration of the $\beta$- stabilization term
\end{tabular}
\end{table}
%
\subsection{Cook's membrane}
The Cook's membrane test consists of a tapered panel fixed along one edge and subject to a shearing load at the opposite edge as depicted in Figure \ref{fig:cooks}.
The applied traction is $t=100$ and $E_t = 250$.
This test problem has no analytical solution.
The vertical tip displacement at corner $C$ is measured, meshing with $2\times 32\times 32$ elements.
\begin{figure}[H]
\centering
\includegraphics[trim={2cm 8cm 2cm 8cm},clip,width=.5\columnwidth]{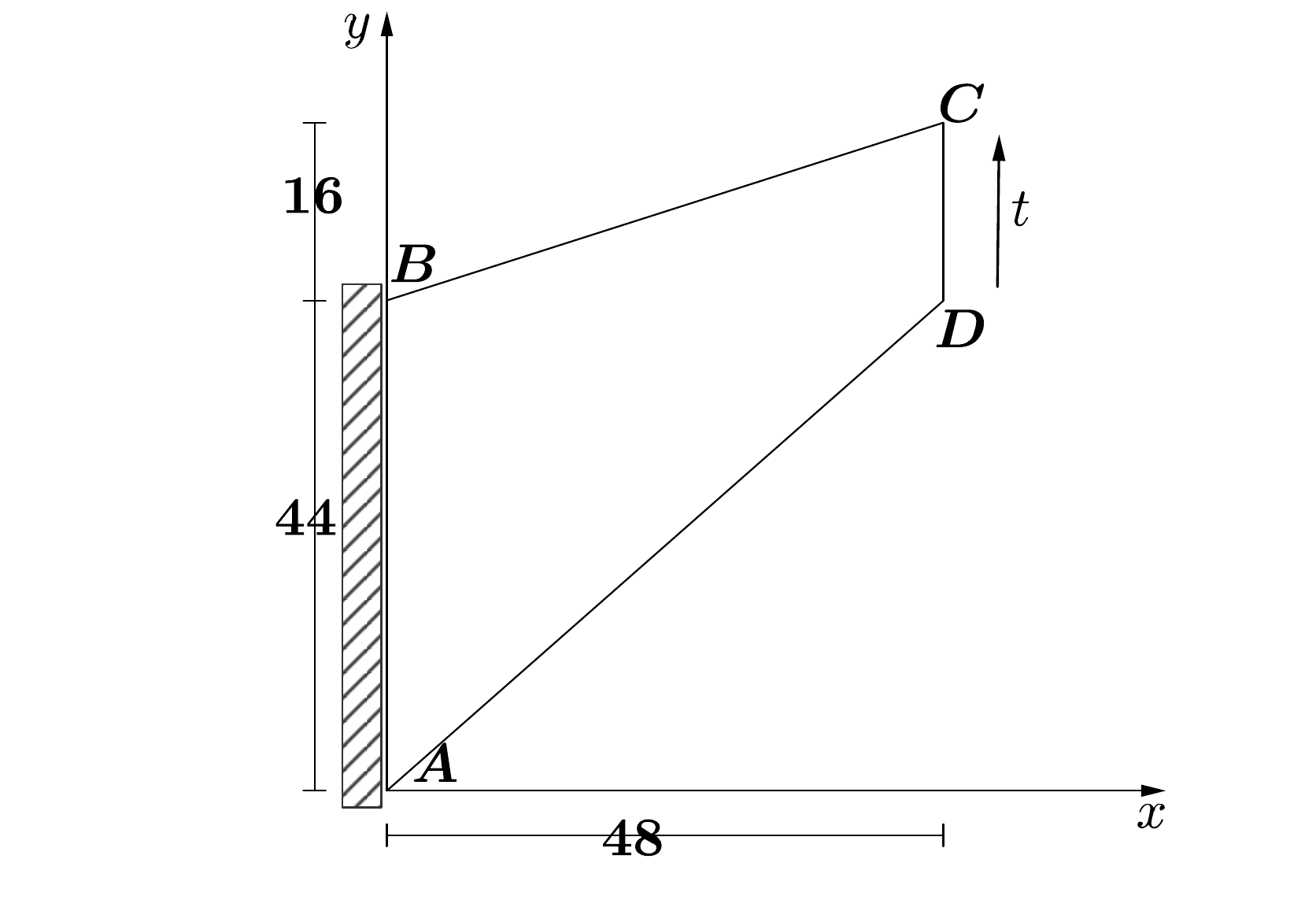}
\caption{Cook's membrane geometry and boundary conditions}
\label{fig:cooks}
\end{figure}
\begin{figure}[h!]
\centering
\subfloat[Moderate values of $p$]{\includegraphics[width=.43\columnwidth]{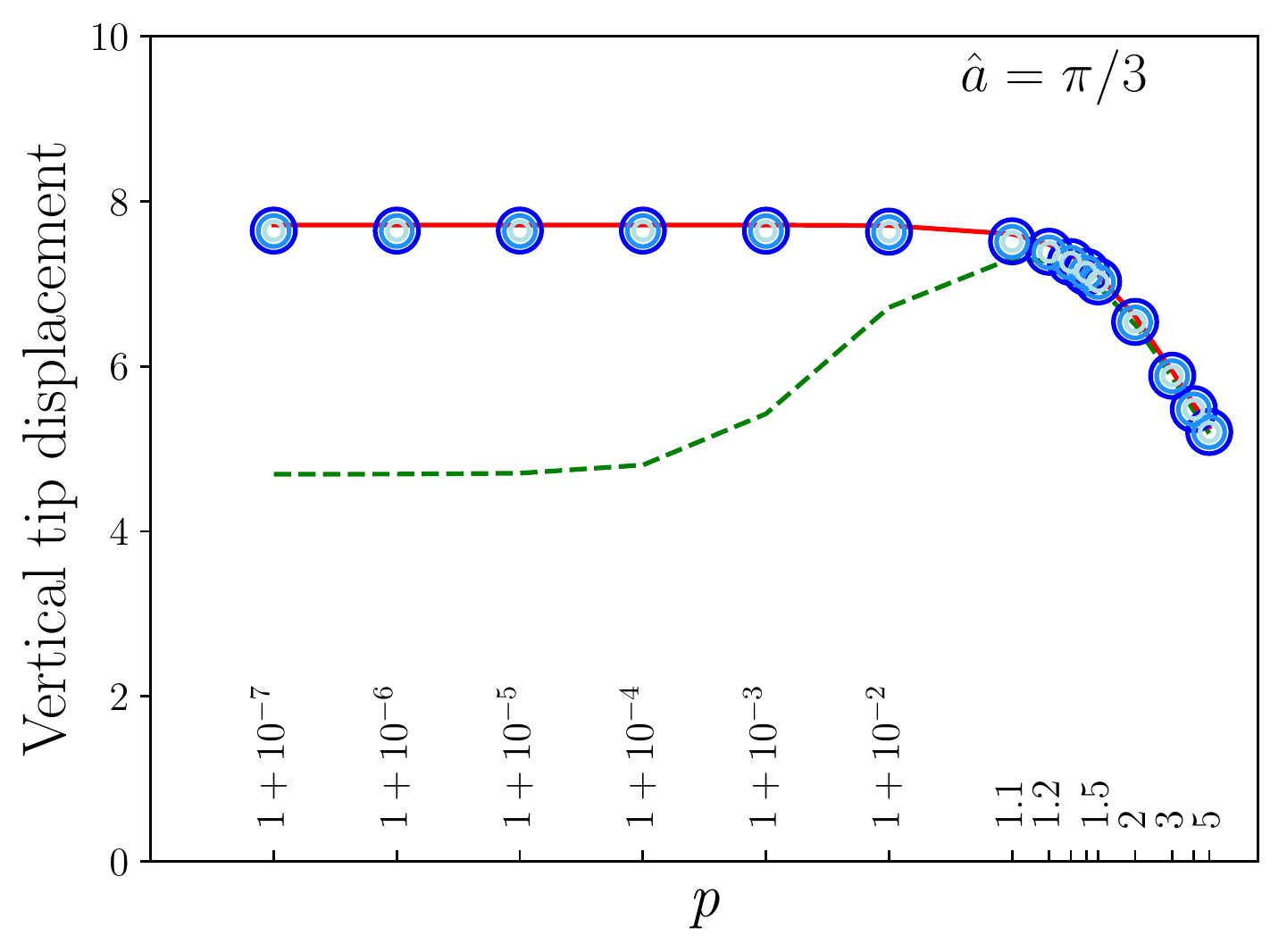}\label{fig:CM_pi3}
								 \includegraphics[width=.57\columnwidth]{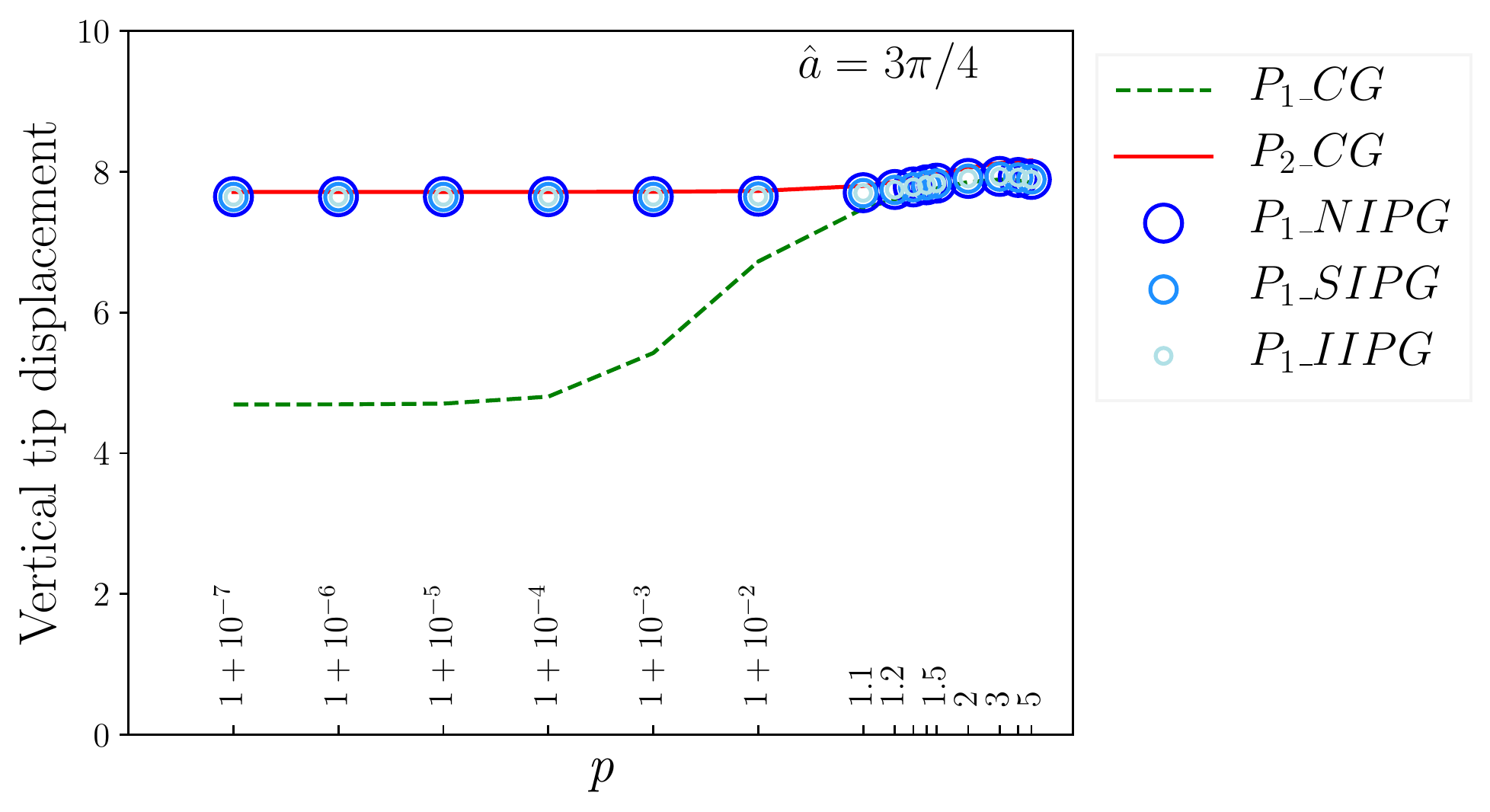}\label{fig:CM_3pi4}}\\
\subfloat[High values of $p$]{\includegraphics[width=.43\columnwidth]{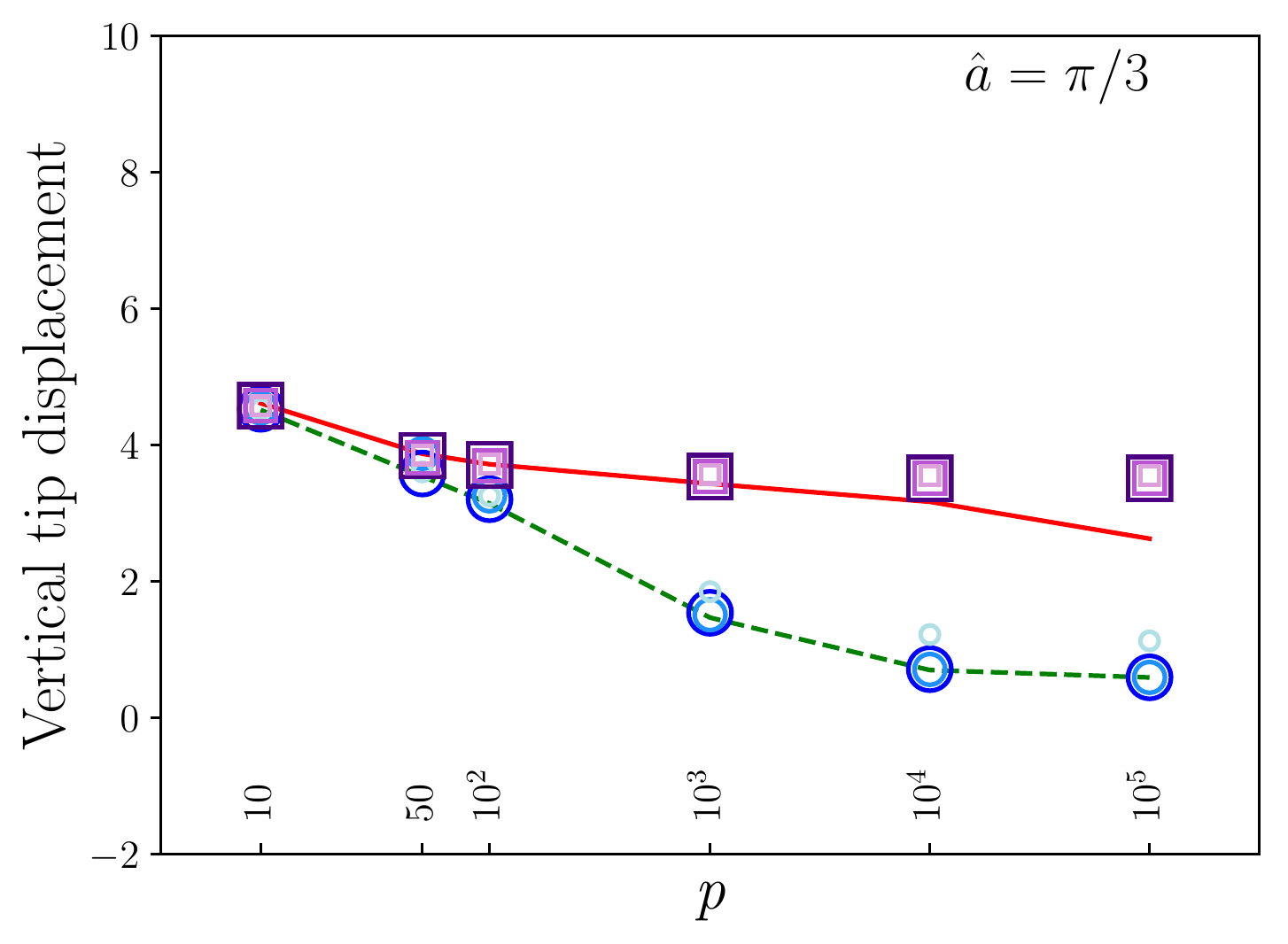}\label{fig:CM_big_pi3}
							 \includegraphics[width=.61\columnwidth]{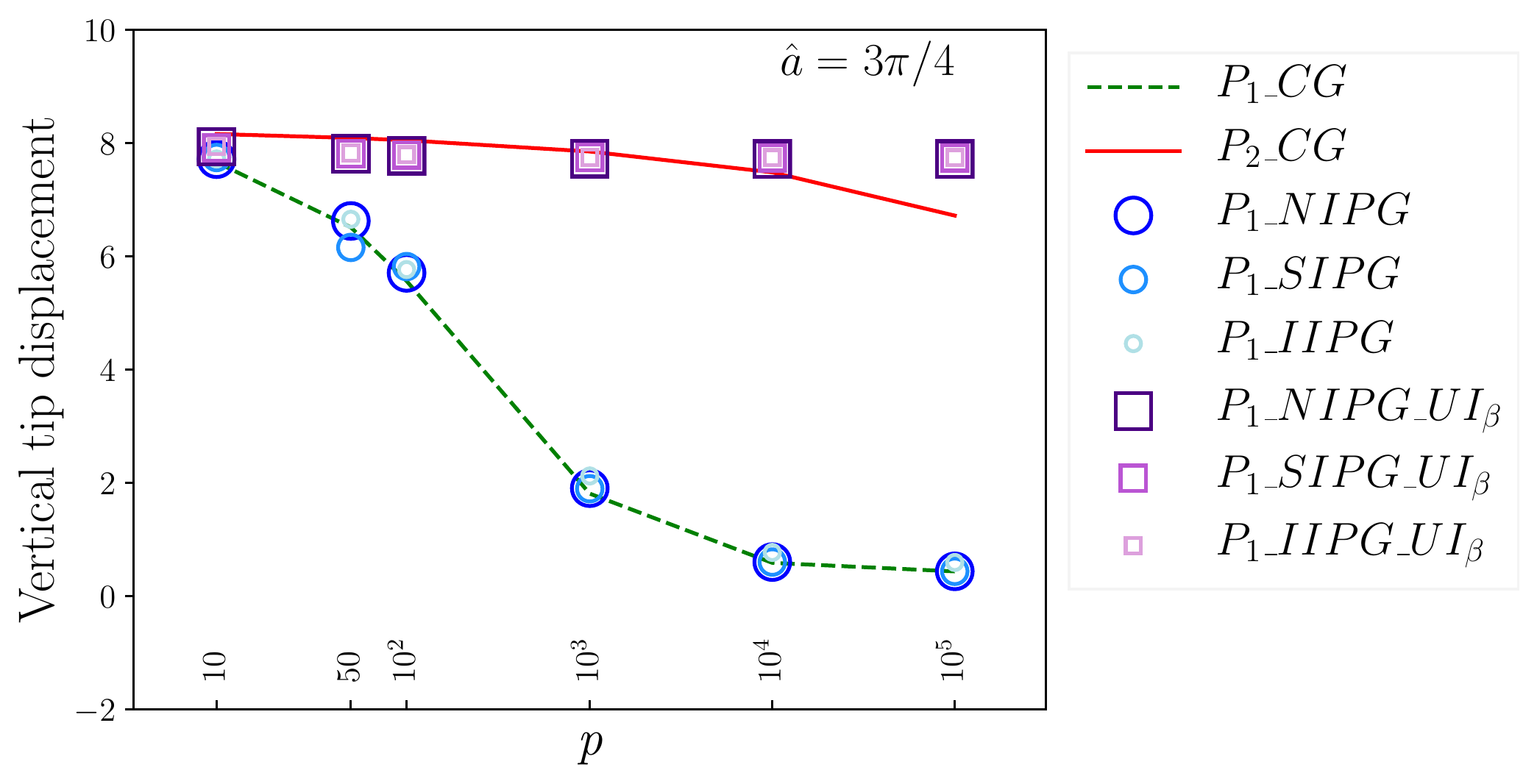}\label{fig:CM_big_3pi4}}
\caption{Tip displacement vs $p$ for the Cook's membrane problem}
\label{fig:CM_tip_disp_log}
\end{figure}

Figure \ref{fig:CM_tip_disp_log} shows semilog plots of the tip displacement vs $p$ for angles $\pi/3$ and $3\pi/4$ for the various element choices,
and for moderate and high values of $p$.
To investigate locking of the proposed formulation, we compare the results with the results obtained using the standard $P_2$-element.

In Figure \ref{fig:CM_pi3}, for moderate values of $p\; (1\leq p \leq 5)$ the $P_1\_CG$ formulation behaves well away from $p=1$, with evidence of locking behaviour as $p\rightarrow 1$.
All three IPDG methods show no locking.
Under-integrated IPDG methods are not necessary here since they give same results.
In Figure \ref{fig:CM_big_pi3}, for higher values of $p\; (10\leq p \leq 10^5)$, the $P_1\_CG$ and all three IPDG methods show locking behaviour as $p$ gets bigger.
Locking is avoided when the $\beta$-stabilization term is under-integrated.
\begin{figure}[h!]
\centering
\includegraphics[width=0.85\columnwidth]{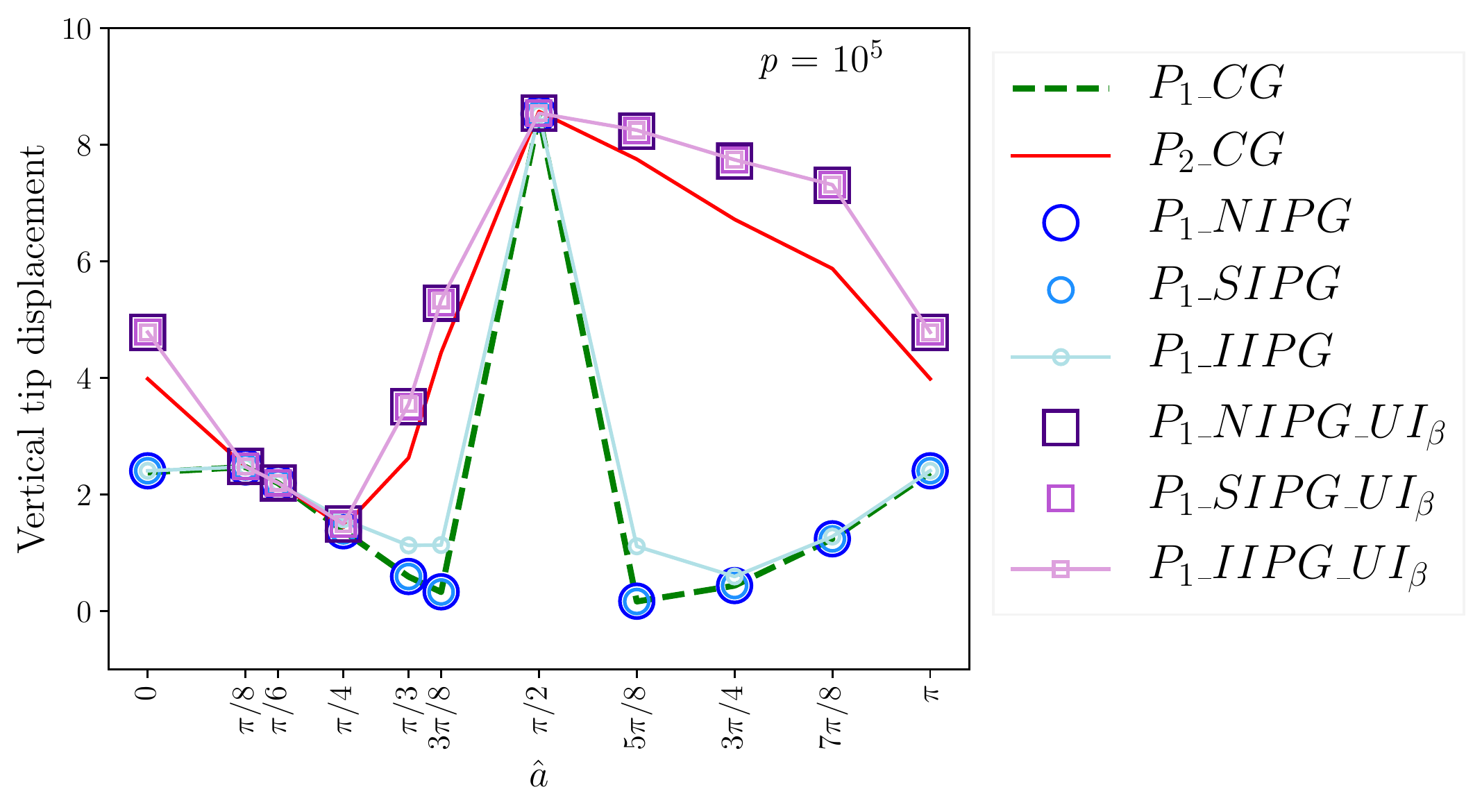}
\caption{Tip displacement for Cook's membrane problem measured at different fibre orientations, for $p=10^5$}
\label{fig:CM_orientation}
\end{figure}
Figure \ref{fig:CM_orientation} shows tip displacements for various fibre orientations, where the degree of anisotropy is fixed at $p=10^5$.
Some deterioration in accuracy is observed for the conforming $P_2$-element, for angles in the range $\hat{a} > \pi/2$.
Whether this indicates mild locking would depend on further parameter-explicit analysis, which is currently absent for the transversely isotropic problem at near-inextensibility.
%
\subsection{Bending of a beam}
\begin{figure}[H]
\centering
\includegraphics[trim={4cm 10cm 0.5cm 11cm},clip,width=.5\columnwidth]{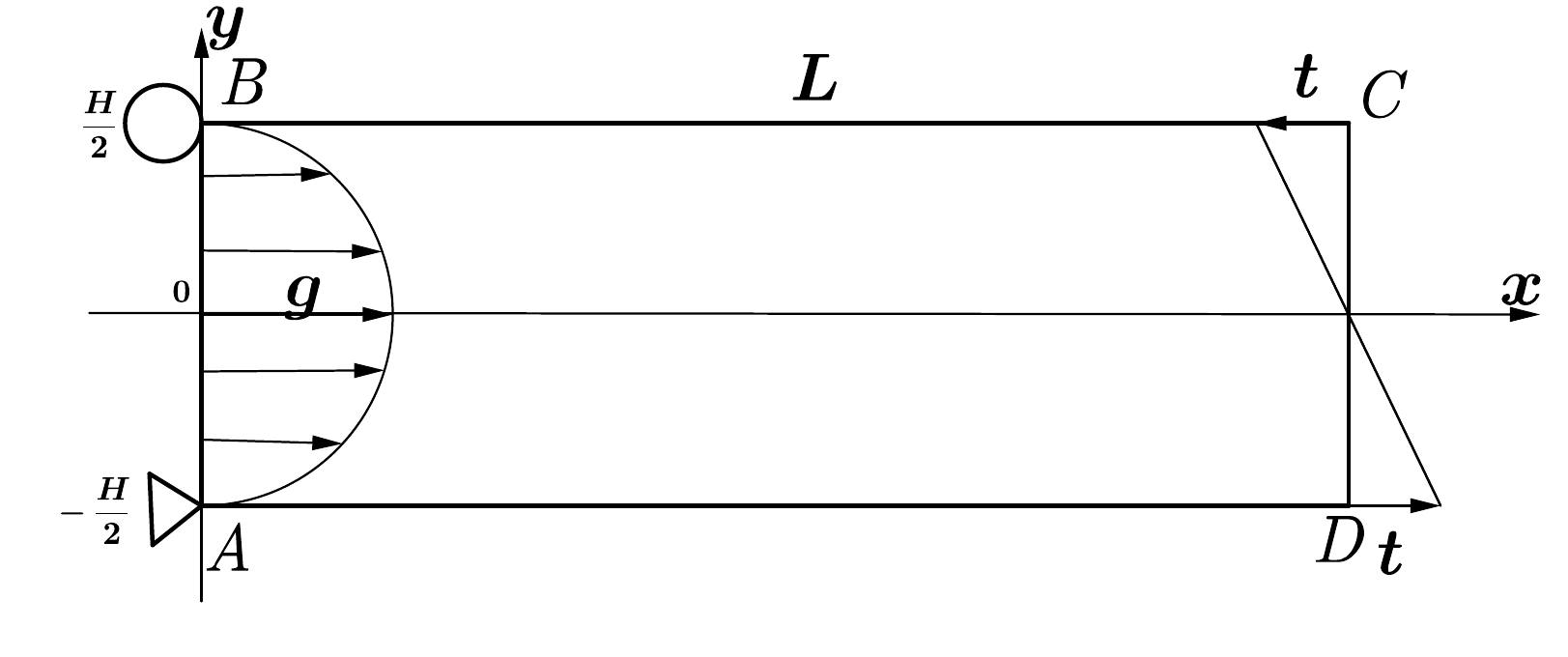}
\caption{Beam geometry and boundary conditions}
\label{fig:bending}
\end{figure}
We consider the beam shown in Figure \ref{fig:bending}, subject to a linearly varying load along the edge CD.
The horizontal displacement $u$ is constrained along A - B, while at A the constraint is in both directions.
The beam has length $L = 10$ and height $H = 2$ and the linearly varying load has a maximum value of $t = 3000$.
Here, $E_t = 1500$.
The boundary conditions are
\[
\begin{dcases}
u(0,y) = g(y),\\
v(0,-\mbox{$\frac{H}{2}$}) = 0,
\end{dcases}
\]
where
\[
g(y) = -\dfrac{t}{H}\mathbb{S}_{31} \left(y^2 - \dfrac{H^2}{4}\right).
\]
The compliance coefficients $\mathbb{S}_{ij}$ are given in \cite{RGR}, as is the analytical solution.
We measure the vertical tip displacement at corner C with meshing $2 \times 80 \times 16$ elements.

\begin{figure}[h!]
\centering
\subfloat[Moderate values of $p$]{\includegraphics[width=.43\columnwidth]{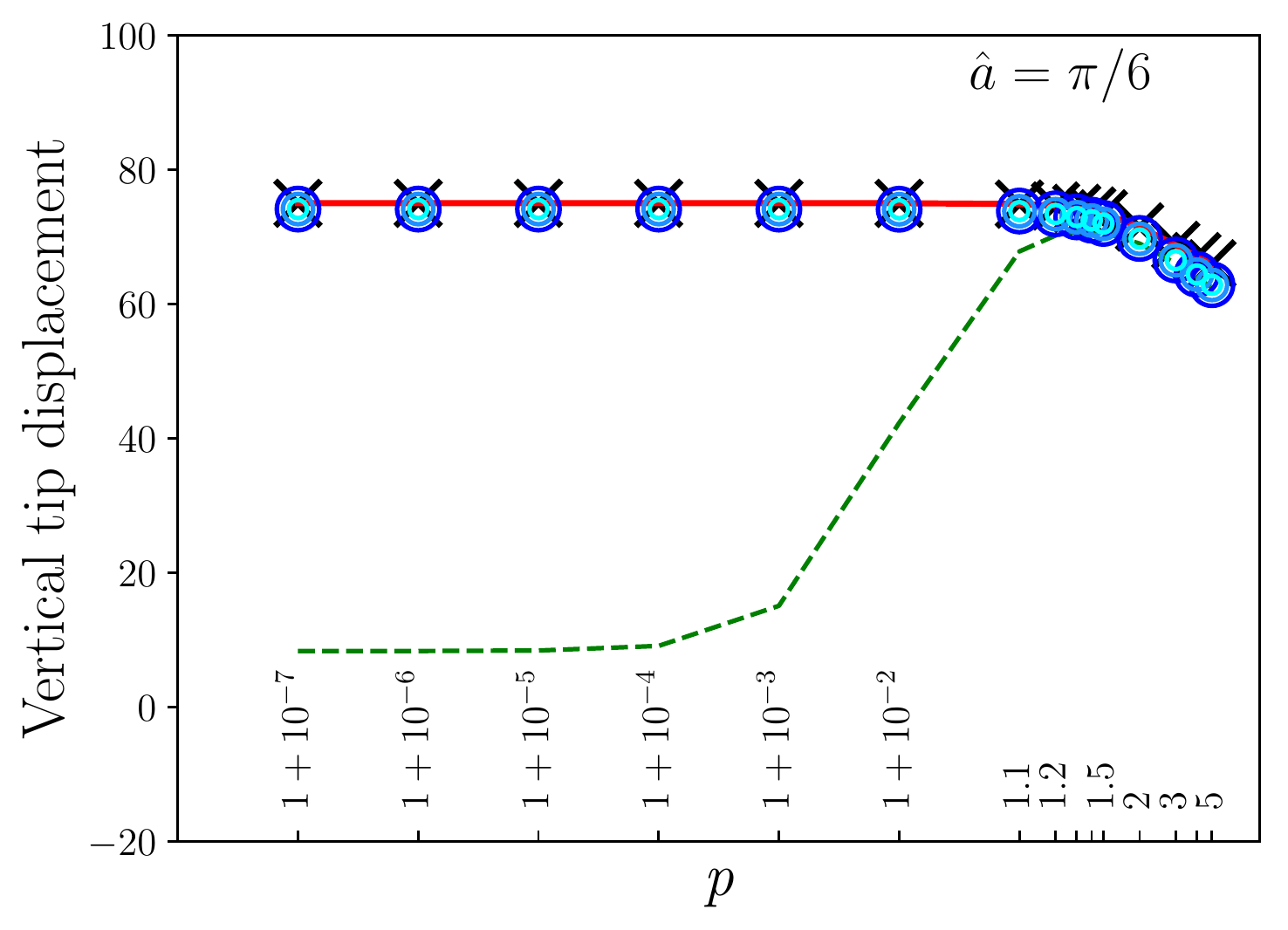}\label{fig:AB_pi6}
								 \includegraphics[width=.57\columnwidth]{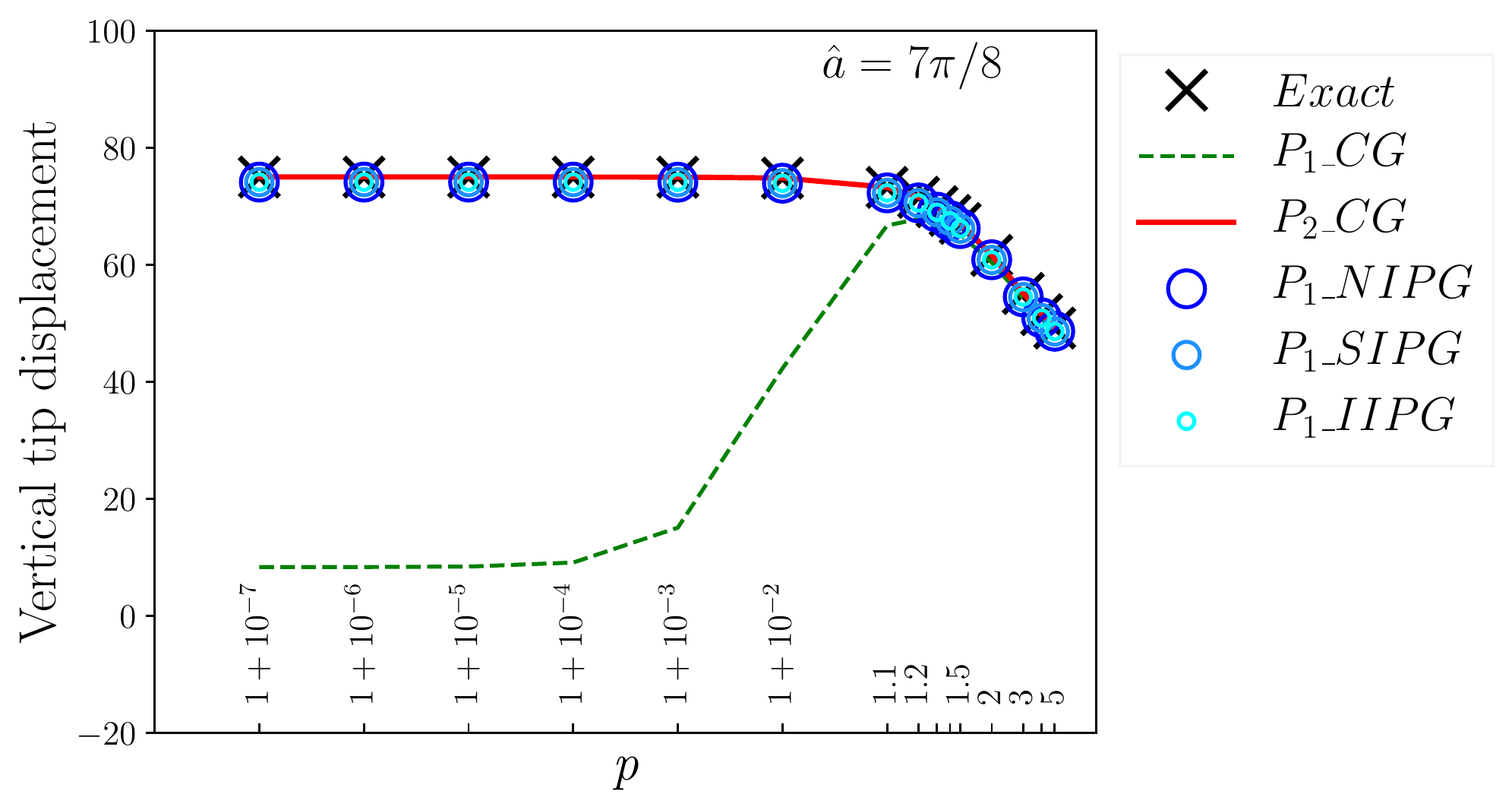}\label{fig:AB_7pi8}}\\
\subfloat[High values of $p$]{\includegraphics[width=.43\columnwidth]{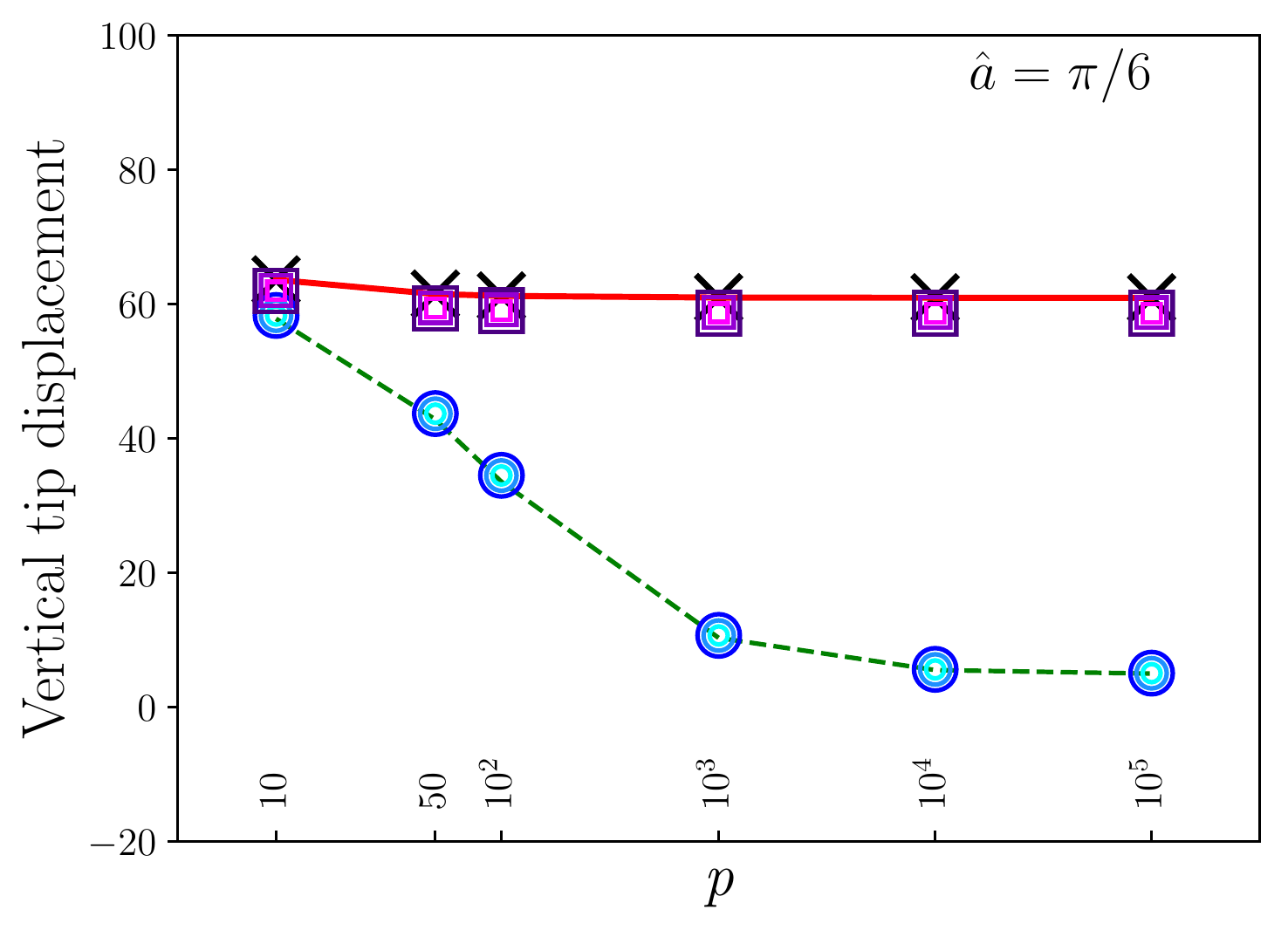}\label{fig:AB_big_pi6}
							 \includegraphics[width=.61\columnwidth]{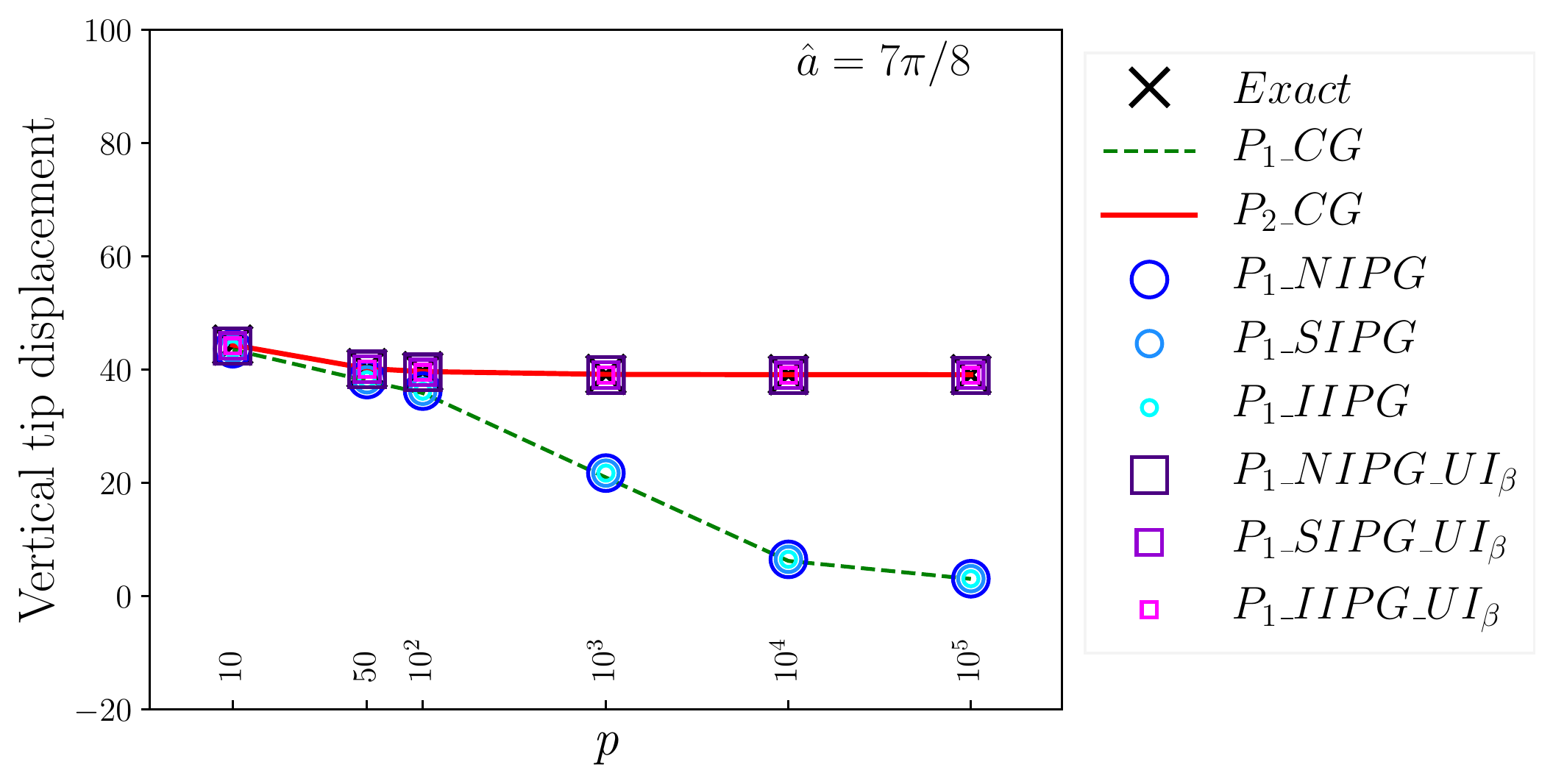}\label{fig:AB_big_7pi8}}
\caption{Tip displacement vs $p$ for the beam problem}
\label{fig:AB_tip_disp_log}
\end{figure}
In Figure \ref{fig:AB_tip_disp_log}, which shows semilog plots of tip displacement for different values of $p$, with the angle of the fibre direction $\pi/6$ and $7\pi/8$, locking behaviour is investigated by comparison with the analytical solution.
The same behaviour as appears for the Cook's example is seen, i.e. for moderate values of $p$ away from $p=1$ (approximately, $1.1 \leq p \leq 5$),
there is locking-free behaviour with $P_1\_CG$, while locking occurs as $p$ approaches $1$. This is overcome by using IPDG methods (Figure \ref{fig:AB_pi6}).
For high values of $p$ there is purely extensional locking with $P_1\_CG$ and all IPDG methods, which is overcome by using under-integration of the $\beta$-stabilization term (Figure \ref{fig:AB_big_pi6}).
\begin{figure}[H]
\centering
\includegraphics[width=0.9\columnwidth]{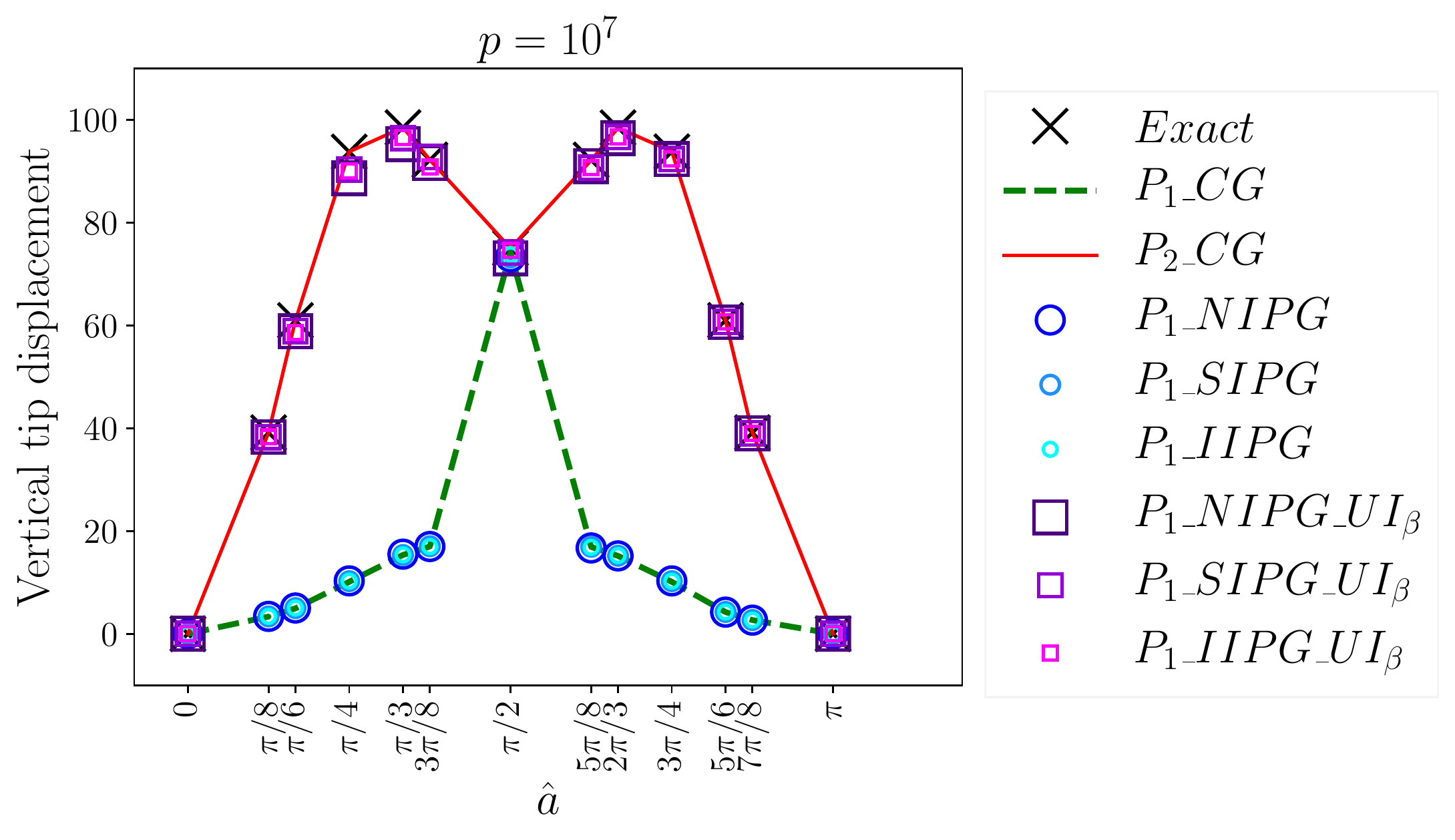}
\caption{Tip displacement for beam problem measured at different fibre orientations, for $p=10^7$}
\label{fig:AB_orientation}
\end{figure}
Figure \ref{fig:AB_orientation} shows tip displacements for various fibre orientations where the degree of anisotropy is fixed at $p=10^7$.
We recall the same behaviour as stated in \cite{RGR}, that is, extensional locking for $P_1\_CG$ except for the angles $0$, where the material is very stiff, and $\pi/2$, where the extensional term tends to $0$.

The following set of results shows behaviour for various fibre orientations, and for values of $p=1.0001, 3$ and $10^4$.

Figure \ref{fig:H1_Error_p1} shows the $\mathcal{H}^1$ relative error convergence plots for all three IPDG formulations and $P_1\_CG$, for $p=1.0001$.
Here all three IPDG formulations show slightly better than optimal (linear) convergence for any fibre direction.
$P_1\_CG$ shows poor convergence, indicative of volumetric locking.
\begin{figure}
\centering
\subfloat[$\hat{a} = \pi/8$]{\includegraphics[width=.52\columnwidth]{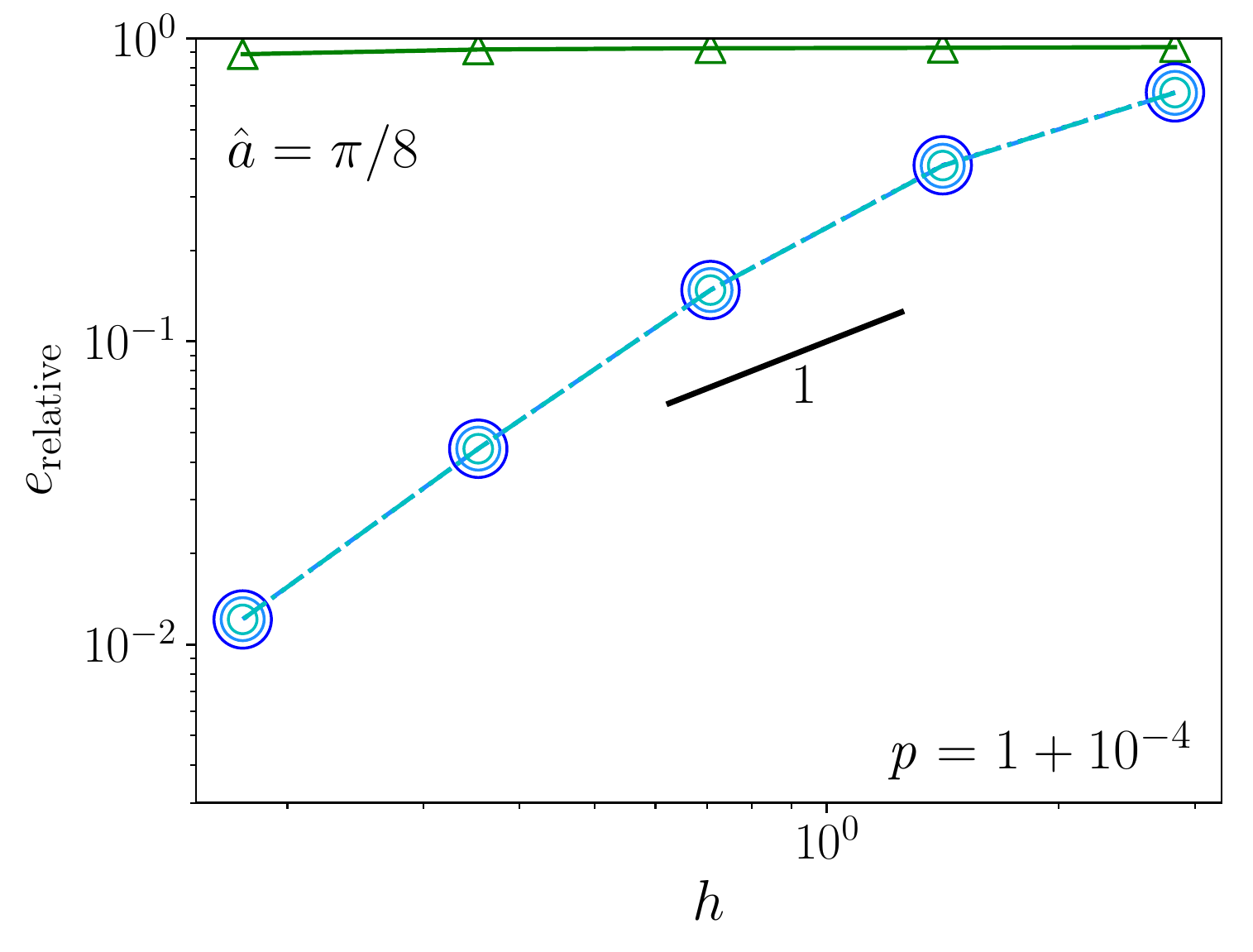}\label{fig:H1_p1_pi8}}
\subfloat[$\hat{a} = \pi/3$]{\includegraphics[width=.52\columnwidth]{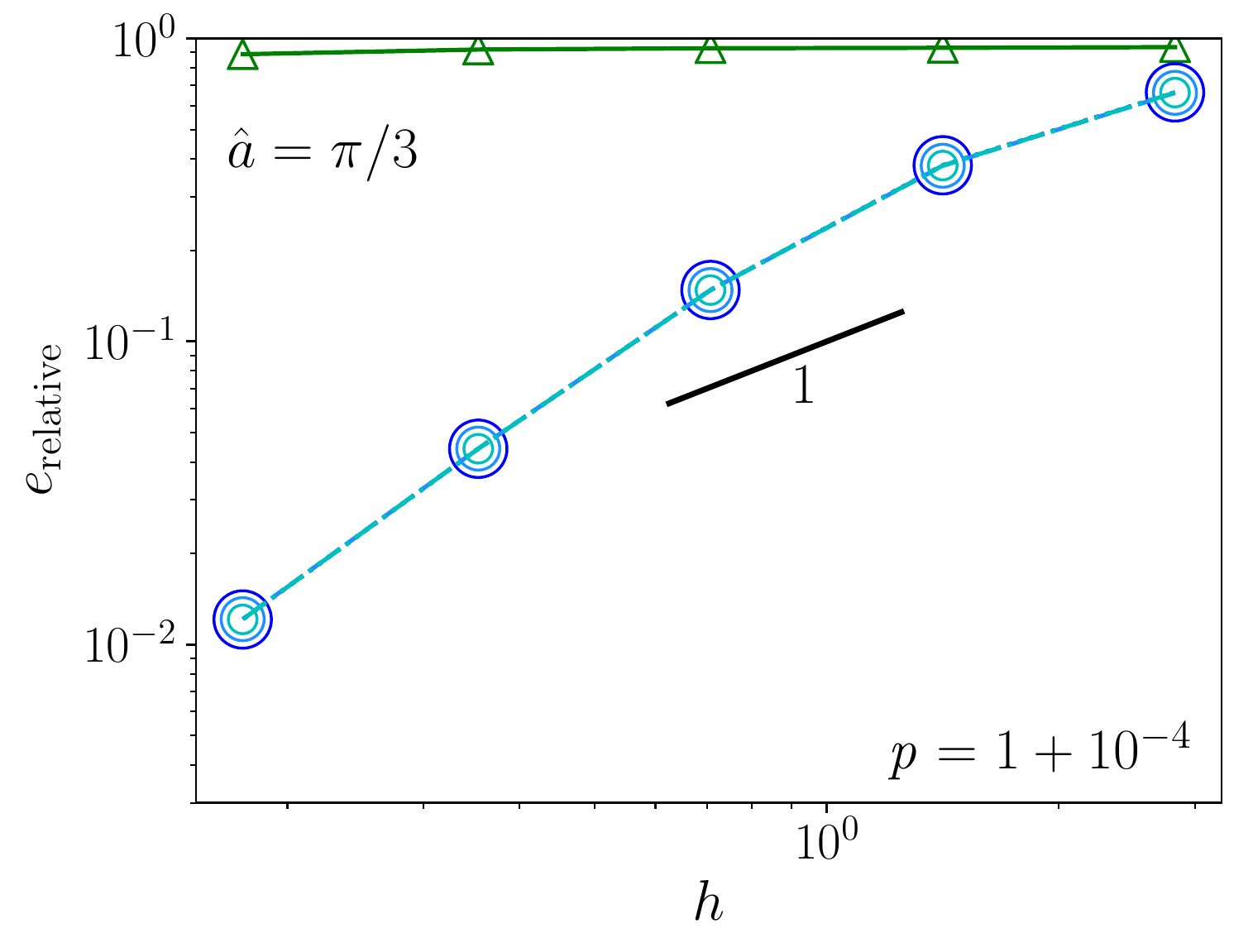}\label{fig:H1_p1_pi3}}\\
\subfloat[$\hat{a} = 5\pi/6$]{\includegraphics[width=.75\columnwidth]{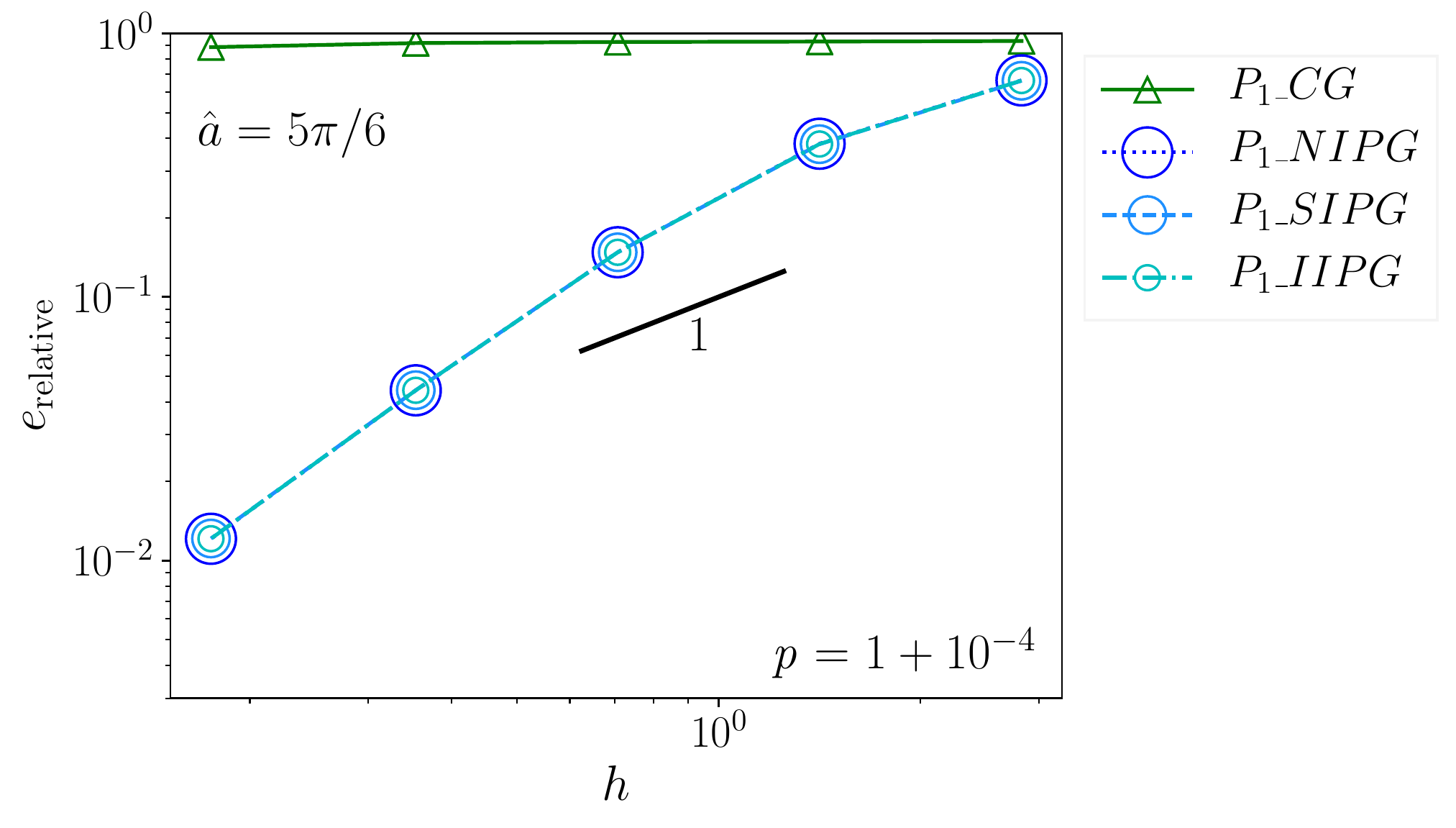}\label{fig:H1_p1_5pi6}}
\caption{Comparison of $\mathcal{H}^1$ relative errors for conforming and full IPDG formulations, for different fibre orientations, and for $p=1.0001$}
\label{fig:H1_Error_p1}
\end{figure}

Figure \ref{fig:H1_Error_p3} shows the $\mathcal{H}^1$ relative error convergence plots for all three IPDG formulations and $P_1\_CG$, for $p=3$.
Here, all formulations at any fibre direction are linearly convergent, indicative of extensional locking.
\begin{figure}
\centering
\subfloat[$\hat{a} = \pi/8$]{\includegraphics[width=.52\columnwidth]{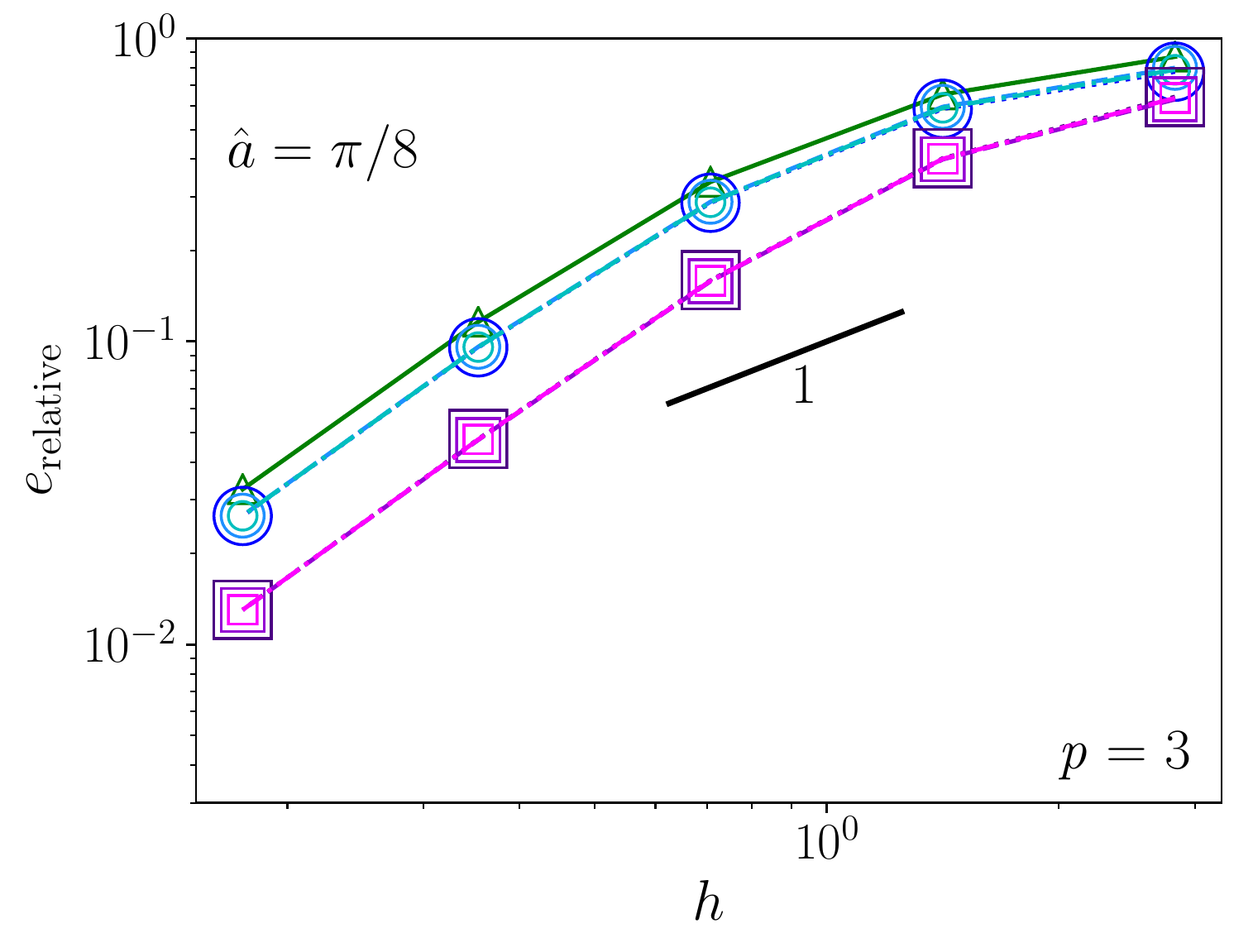}\label{fig:H1_p3_pi8}}
\subfloat[$\hat{a} = \pi/3$]{\includegraphics[width=.52\columnwidth]{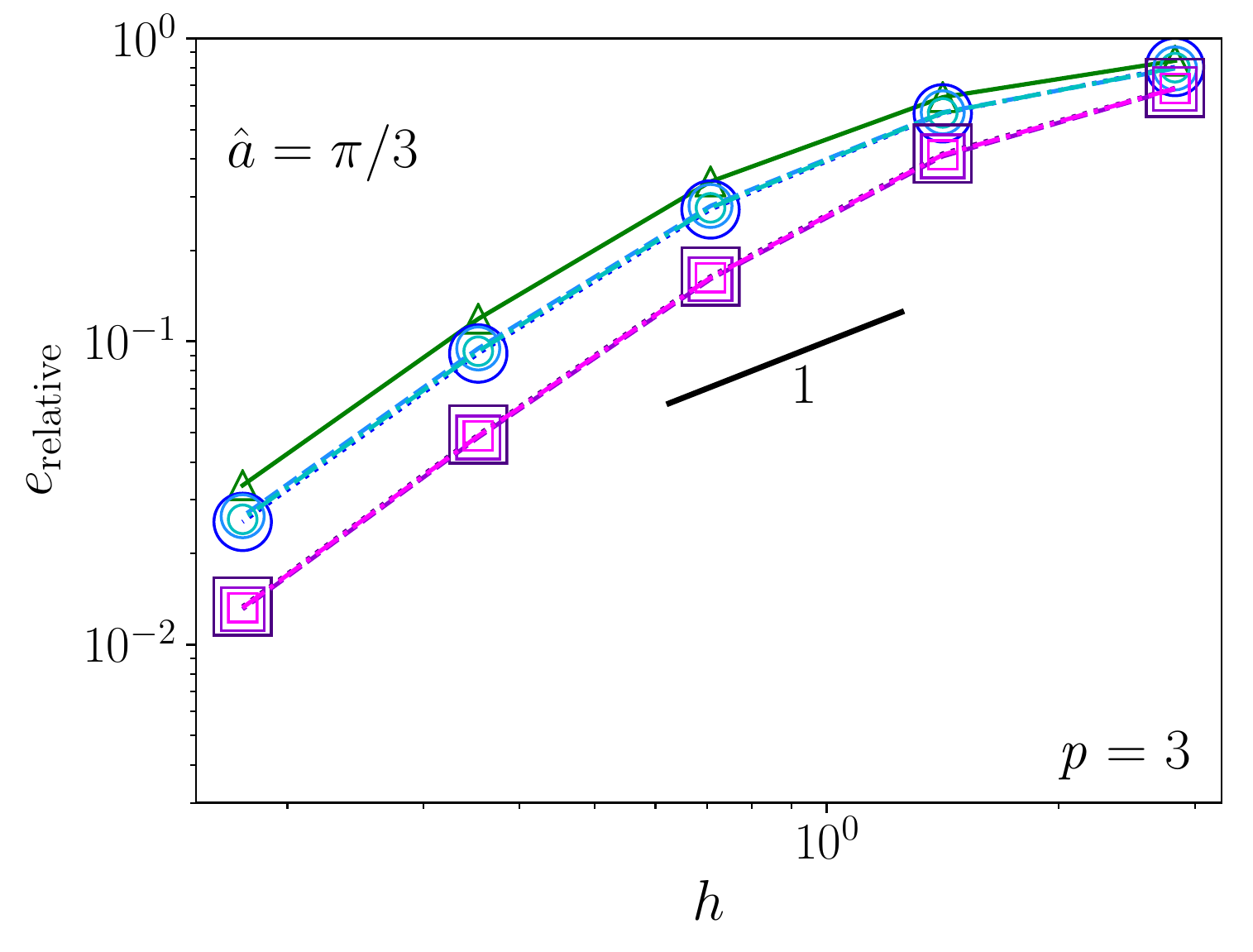}\label{fig:H1_p3_pi3}}\\
\subfloat[$\hat{a} = 5\pi/6$]{\includegraphics[width=.75\columnwidth]{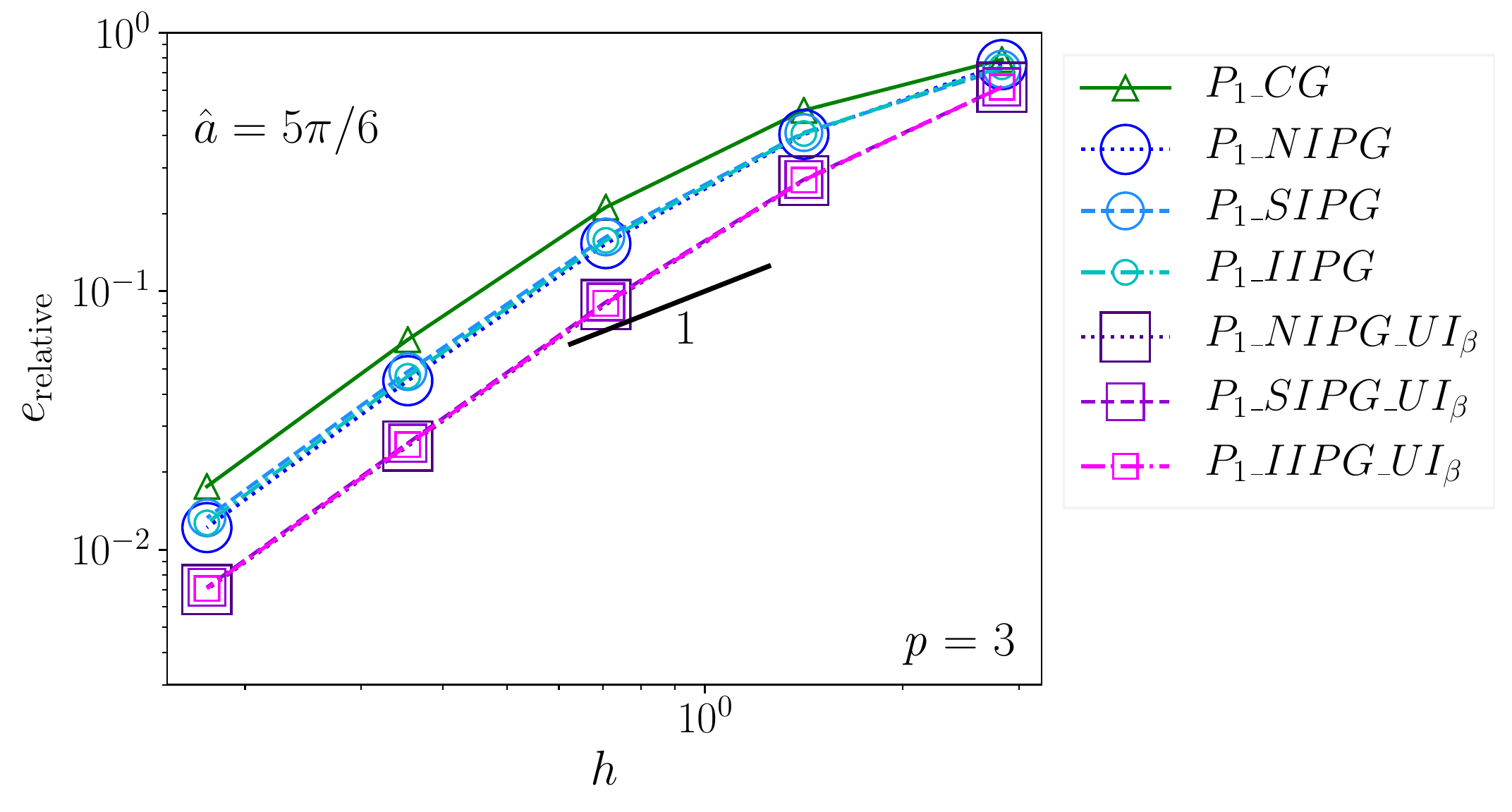}\label{fig:H1_p3_5pi6}}
\caption{Comparison of $\mathcal{H}^1$ relative errors for conforming and full IPDG formulations, for different fibre orientations, and for $p=3$}
\label{fig:H1_Error_p3}
\end{figure}

Figure \ref{fig:H1_Error_p1e4} shows the $\mathcal{H}^1$ relative error convergence plots for all three IPDG formulations and $P_1\_CG$, for $p=10^4$.
$P_1\_CG$ and the full IPDG methods show poor convergence. All three IPDG methods with under-integration of the $\beta$-stabilization term show optimal convergence at rate $1.6$.
\begin{figure}
\centering
\subfloat[$\hat{a} = \pi/8$]{\includegraphics[width=.52\columnwidth]{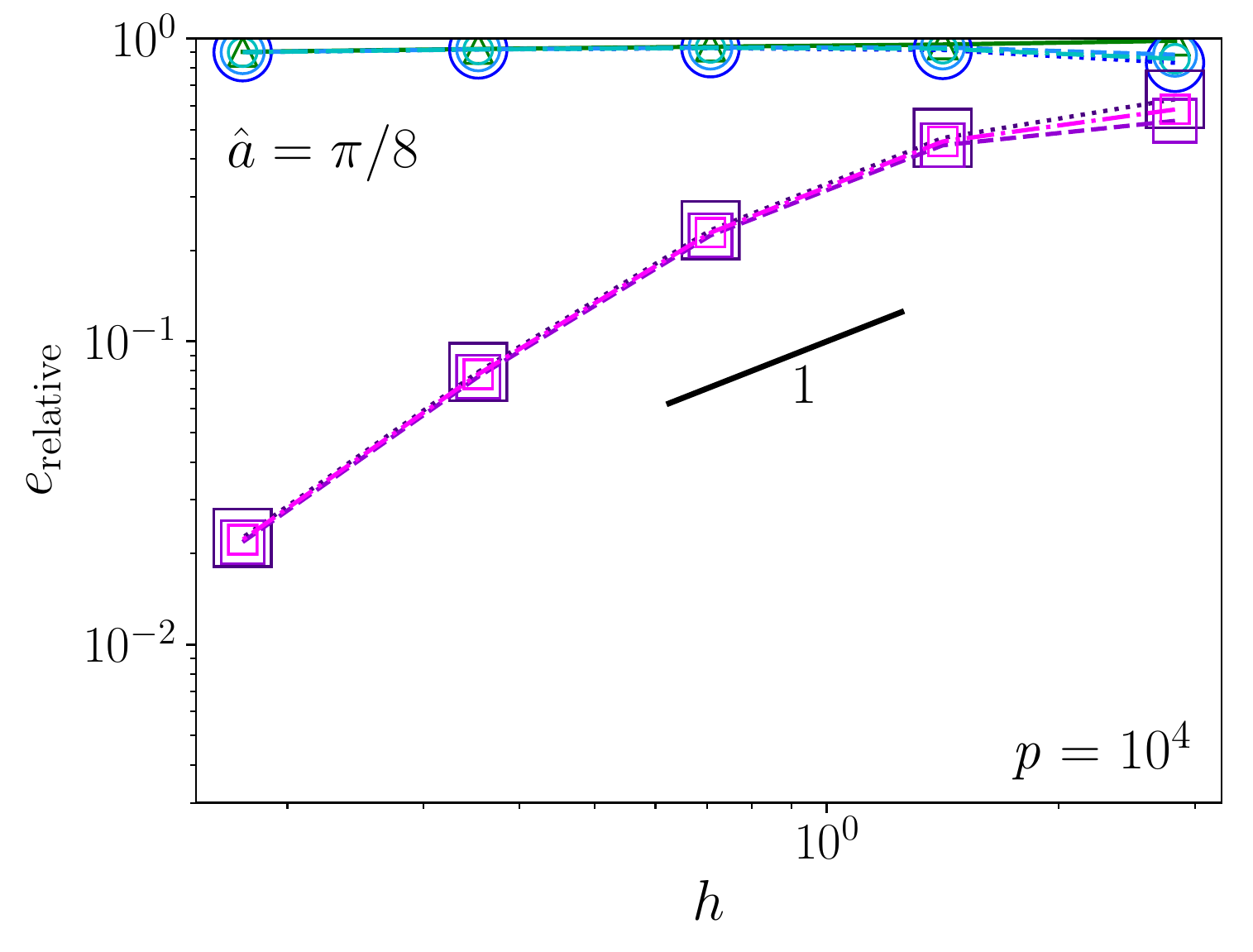}\label{fig:H1_p1e4_pi8}}
\subfloat[$\hat{a} = \pi/3$]{\includegraphics[width=.52\columnwidth]{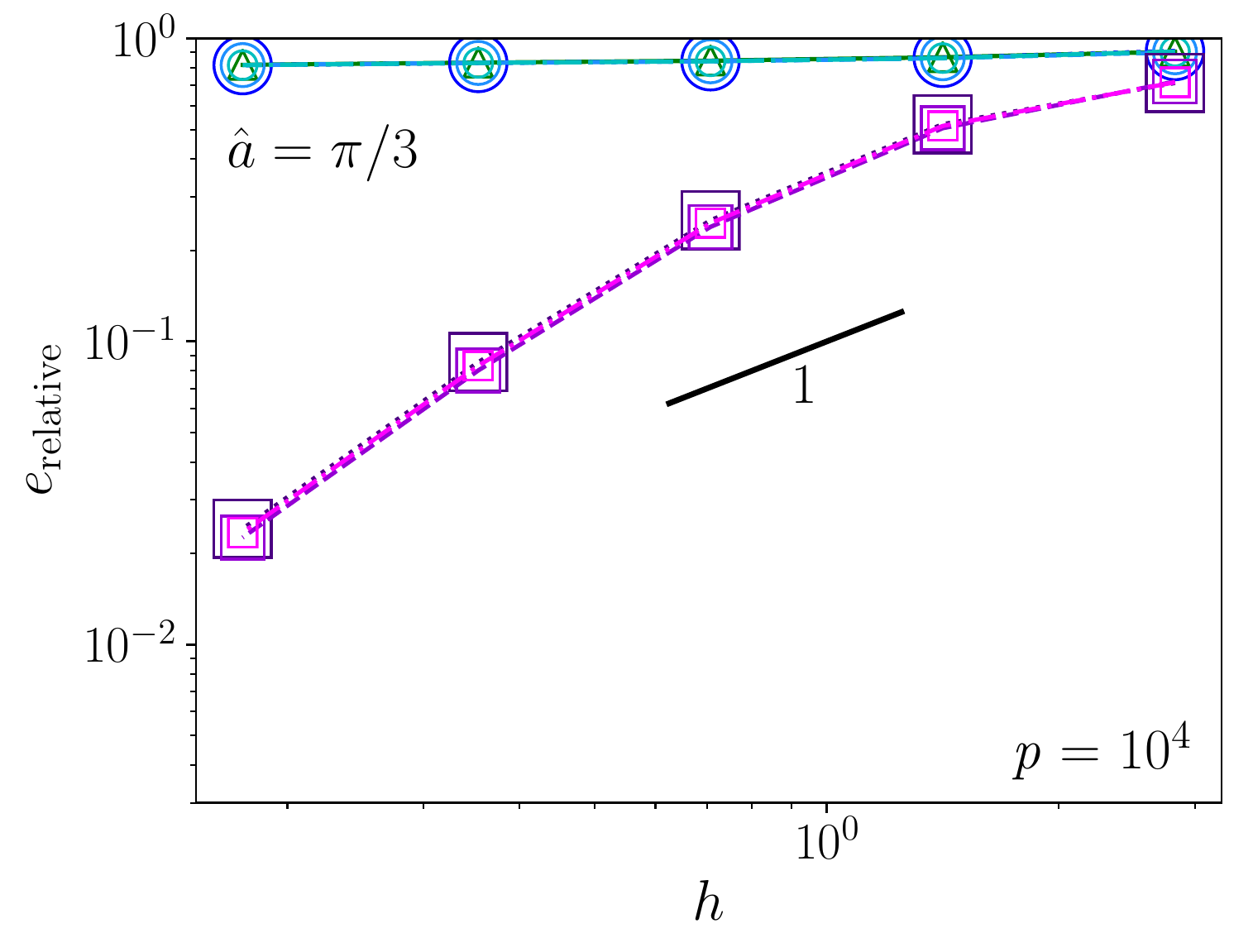}\label{fig:H1_p1e4_pi3}}\\
\subfloat[$\hat{a} = 5\pi/6$]{\includegraphics[width=.75\columnwidth]{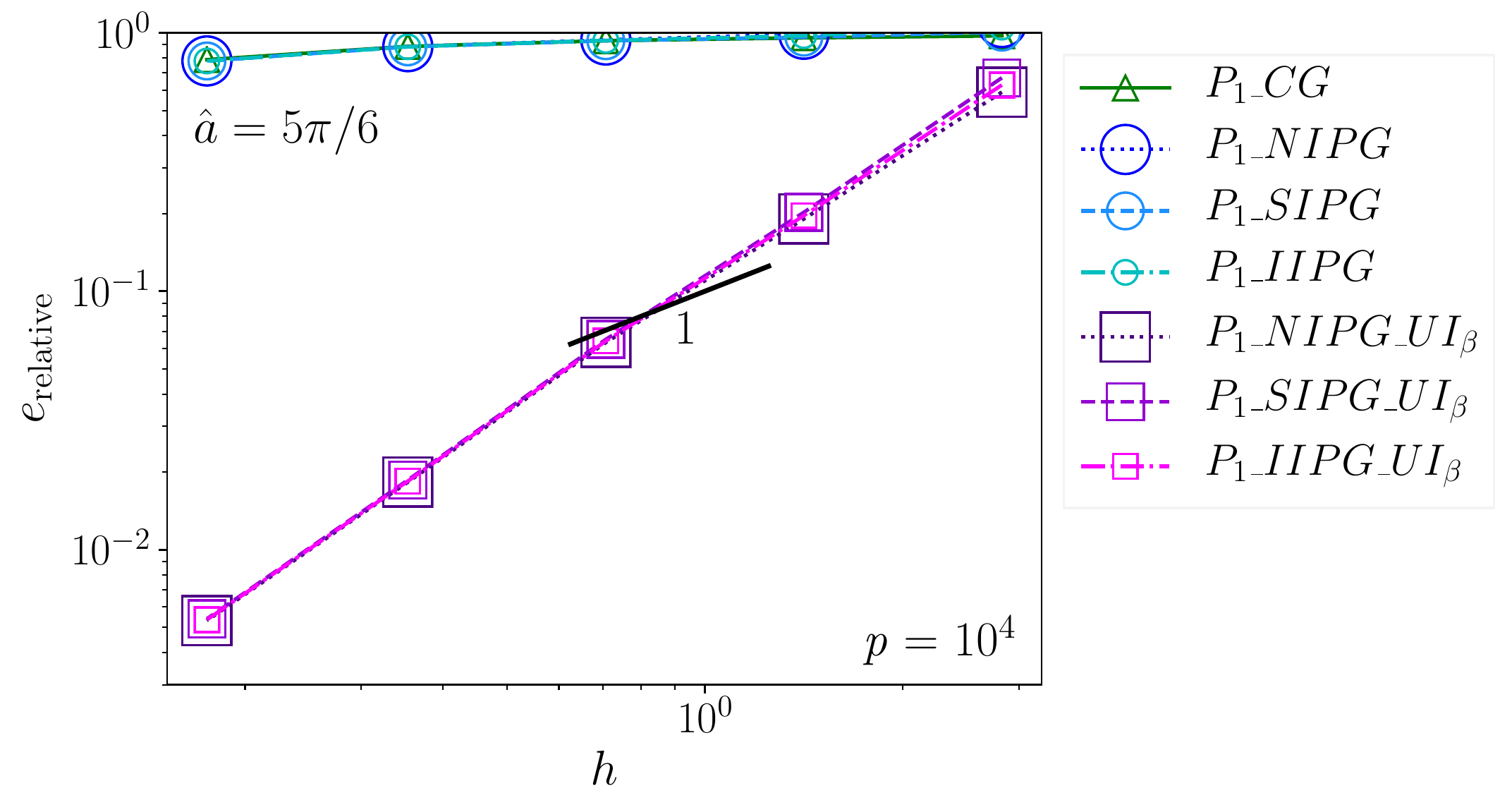}\label{fig:H1_p1e4_5pi6}}
\caption{Comparison of $\mathcal{H}^1$ relative errors for conforming, full IPDG and under-integrated IPDG formulations, for different fibre orientations, and for $p=10^4$}
\label{fig:H1_Error_p1e4}
\end{figure}

\section{Conclusions}
This paper has presented new IPDG formulations for transversely isotropic linear elasticity, and their analyses.
It has shown that these IPDG methods are volumetric locking-free while using low-order triangular meshes.
However, locking has been shown to manifest at the near-inextensible limit: this is suggested by an analytical investigation of the error bound, and confirmed by numerical investigations.
The use of under-integration in the extensional stabilization term of the IPDG formulation is proposed as a remedy.
The analysis presented, assuming an a priori estimate analogous to that of Brenner and Sung for the isotropic case, proves that it is uniformly convergent with respect to the extensibility parameter for low-order triangular elements.
Supporting numerical results over a range of measures of anisotropy and a range of fibre directions show the locking-free behaviours.

A corresponding study using conforming finite element approximations was presented in previous work \cite{RGR}.
This work is presented as an extension to that by using discontinuous Galerkin approaches.
Both works are intended to enhance current understanding of the cases of near-incompressibility and near-inextensibility.
 Further investigations will be directed towards the study of displacement-based formulations for problems posed on non-cartesian coordinate systems, non-homogeneous and large deformation problems, as well as mixed formulations.
%
\section*{Acknowledgements}
The authors are grateful to a referee for comments that have led to an improvement in this work. The authors acknowledge with thanks the support for this work by the National Research Foundation, through the South African Research Chair in Computational Mechanics.
\bibliographystyle{plain}
\bibliography{Main.bib} 


\end{document}